\documentclass[10pt,a4paper,leqno]{amsart}
\usepackage{amsmath}
\usepackage{amssymb}
\usepackage{amsfonts}
\usepackage{amsthm}
\usepackage{mathrsfs}
\usepackage{dsfont}
\usepackage{mathtools}
\usepackage{bm}
\usepackage{graphicx}
\usepackage{todonotes}
\usepackage{nicefrac}
\usepackage{esint}
\usepackage{bbm}		
\usepackage{enumitem}
\usepackage[backref=page]{hyperref}
\usepackage{color}

\definecolor{darkgreen}{rgb}{0.5,0.25,0}
\definecolor{darkblue}{rgb}{0,0,1}
\definecolor{answerblue}{rgb}{0,0,0.75}

\hypersetup{colorlinks,breaklinks,
linkcolor=darkblue,urlcolor=darkblue,
anchorcolor=darkblue,citecolor=darkblue}

\newcommand*{\mailto}[1]{\href{mailto:#1}{\nolinkurl{#1}}}

\renewcommand{\d}{\mathrm{d}}

\newcommand{\cell}{{\mathcal{C}}}
\DeclareMathOperator{\diag}{diag}

\newcommand{\R}{\mathbb{R}}
\newcommand{\T}{{\mathbb{T}}}
\newcommand{\N}{\mathbb{N}}
\newcommand{\Z}{\mathbb{Z}}

\newcommand{\bk}[1]{ \left(  #1 \right)}
\newcommand{\abs}[1]{ \mathopen|#1\mathclose|}
\newcommand{\norm}[1]{ \mathopen\|#1\mathclose\|}
\newcommand{\one}[1]{\mathds{1}_{#1}}
\newcommand{\Dx}{{\Delta x}}

\newcommand{\hf}{{\nicefrac12}}
\newcommand{\thf}{{\nicefrac32}}
\newcommand{\ind}{\mathbbm{1}}	
\renewcommand{\le}{\leqslant}
\renewcommand{\leq}{\leqslant}
\renewcommand{\ge}{\geqslant}
\renewcommand{\geq}{\geqslant}
\renewcommand{\phi}{ \varphi}
\newcommand{\from}{\colon}

\newcommand{\dconv}{\mathbin{\bm{*}}} 

\newcommand{\relspace}{\hphantom{{}={}}}

\theoremstyle{theorem}
\newtheorem{theorem}{Theorem}[section]
\newtheorem{proposition}[theorem]{Proposition}
\newtheorem{lemma}[theorem]{Lemma}
\newtheorem{corollary}[theorem]{Corollary}

\theoremstyle{definition}
\newtheorem{definition}[theorem]{Definition}

\theoremstyle{theorem}
\newtheorem{remark}[theorem]{Remark}

\allowdisplaybreaks
\numberwithin{equation}{section}

\title[Semi-discrete heat equations with variable coefficients]
{Semi-discrete heat equations with variable coefficients
and the parametrix method}

\author[Fjordholm]{U.S. Fjordholm}
\address[Ulrik S. Fjordholm]{Department of Mathematics\\
   University of Oslo\\
  NO-0316 Oslo\\ Norway}
\email{\mailto{ulriksf@math.uio.no}}

\author[Karlsen]{K.H. Karlsen}
\address[Kenneth H. Karlsen]{Department of Mathematics\\
   University of Oslo\\
  NO-0316 Oslo\\ Norway}
\email{\mailto{kennethk@math.uio.no}}

\author[Pang]{P.H.C. Pang}
\address[Peter H.C. Pang]{Department of Mathematics\\
   University of Oslo\\
  NO-0316 Oslo\\ Norway}
\email{\mailto{ptr@math.uio.no}}



\thanks{We gratefully acknowledge the
support of the Research Council of Norway
through the projects INICE (301538)
and NASTRAN (351123).}

\date{\today}

\begin{document}

\begin{abstract}
We develop a parametrix approach for constructing
solutions and establishing grid size independent estimates
for semi-discrete heat equations with variable coefficients.
While the classical continuous setting benefits from Gaussian estimates
of the constant coefficient heat kernel, such estimates
are not available in the semi-discrete context.
To address this complication, we derive estimates involving products
of heavy-tailed Lorentz (also known as Cauchy) probability densities.
These Lorentzian estimates provide a sufficient
handle on certain iterated convolutions involving
Bessel functions, enabling us to
achieve convergence of the parametrix approach.
\end{abstract}

\maketitle


\section{Introduction}\label{sec:intro}
In this paper, we develop the parametrix approach
to construct solutions and establish grid size independent
estimates for the semi-discrete variable
coefficient heat equation
\begin{align}\label{eq:heat_equation1-intro}
	\frac{\d}{\d t}  u_\alpha(t)
	= \sum_{j = 1}^d c_\alpha^j \nabla_+^j
	\nabla_-^j  u_\alpha(t) + f_\alpha(t),
	\qquad \alpha \in \Z^d, \,\, d\ge 1.
\end{align}
Here, $\nabla_\pm^j$ denote the operators
of forward and backward differences in the $j$th direction,
defined on the grid $\Delta x \mathbb{Z}^d$, where
$\Dx > 0$ is the grid size parameter, $j \in \{1, \ldots, d\}$.
The coefficients $\{c_{\alpha}\}_{\alpha \in \mathbb{Z}^d}$
are positive and uniformly bounded away from
zero and infinity, while $\{f_\alpha\}_{\alpha
\in \Z^d}$ represents an inhomogeneity
with certain summability properties over
$\Delta x \Z^d\times[0,\infty)$.

The parametrix method in the continuous case
is a classical technique used to solve
parabolic equations with variable coefficients
\cite{Fri1964,Taylor:2011aa}. Consider the parabolic equation
$\partial_t u= A(x)u + f$, where $A(x)$ is a second-order
differential operator with variable coefficients, for example
$A(x)=\tfrac12c(x)\Delta$. To solve the 
Cauchy problem using the parametrix method,
one starts by constructing an initial
approximation $\Gamma_0$ (the parametrix)
to the fundamental solution
$\Gamma$ for $\partial_t-A(x)$.
Denote by $E$ the error, so that $\Gamma = \Gamma_0 + E$.
Then $E$ will satisfy the equation
\begin{equation}\label{eq:intro-error}
	\bigl( \partial_t - A(x) \bigr) E
	= -\bigl( \partial_t - A(x) \bigr) \Gamma_0 \eqqcolon K,
\end{equation}
which gives $E = \Gamma * K$, where $*$ refers
either to spatial or spatio-temporal convolution.
By substituting $\Gamma = \Gamma_0 + E$ into this
expression, we obtain
$$
E = \Gamma_0 * K + E * K
\quad \text{or} \quad \bigl(I - * K \bigr) E
= \Gamma_0 * K.
$$
The formal solution to this equation
is given by a Neumann series:
$$
E = \bigl(I - * K\bigr)^{-1}(\Gamma_0 * K)
= \sum_{n=1}^\infty (* K)^n(\Gamma_0),
$$
where $(* K)^n$ represents the
$n$-fold convolution operator.
If this series representation of the error $E$
converges, we obtain the required
function $\Gamma = \Gamma_0 + E$:
\begin{equation}\label{eq:intro-parametrix}
	\Gamma = \sum_{n=0}^\infty \Gamma_n, \qquad
	\Gamma_n\coloneqq(* K)^n(\Gamma_0).
\end{equation}
Given (the fundamental solution) $\Gamma$, the solution
to the variable coefficient
Cauchy problem for $\partial_t u = A(x)u + f$, with initial
condition $u(0) = u_0$, can be expressed
using Duhamel's formula.

How is $\Gamma_0$ is chosen? For the heat equation with
variable coefficient $c(\cdot)>0$,
i.e., $A(x)=\tfrac12c(x)\Delta$, define
the operator $A_{x_0}$ by ``freezing the coefficient" of $A$
at a point $x_0$: $A_{x_0}=\tfrac12c_0\Delta$, $c_0=c(x_0)$,
and then choose $\Gamma_0$ as the fundamental solution
of $\partial_t-A_{x_0}$, a Gaussian heat kernel:
\begin{equation}\label{eq:heat-kernel-cont}
	\Gamma_0(t -\tau,x,y) = \frac{1}{(4\pi c_0 (t-\tau))^{d/2}}
	\exp\left(-\frac{|x-y|^2}{4c_0(t-\tau)} \right).
\end{equation}
In this case, the right-hand side $K$ of the error equation
\eqref{eq:intro-error} simplifies into
\begin{equation}\label{eq:intro-Phi}
	K = \bigl(A(x) - A_{x_0}\bigr) \Gamma_0
	=\bigl(c(x)-c(x_0)\bigr) \Delta \Gamma_0.
\end{equation}

The parametrix approach is purely formal;
to make it rigorous, one must prove that the
series \eqref{eq:intro-parametrix} is convergent.
Technically, this is achieved by deriving Gaussian estimates
for $\Gamma_n$ (and for its derivatives):
\begin{equation}\label{eq:intro-Gaussian}
	\abs{\Gamma_n(t - \tau,x,y)} \leq
	\frac{C_1}{(t-\tau)^{d/2-\nu}}
	\exp\left(-C_2\frac{|x-y|^2}{(t-\tau)}\right),
	\quad n\in \N,
\end{equation}
where $C_1,C_2>0$ and $\nu\in \N$.
For details, see \cite[Ch.~9]{Fri1964} or
\cite[p.~56]{Taylor:2011aa}.

In this paper, we continue the work initiated in our previous work~\cite{Fjordholm:2023aa} and develop the parametrix approach in the
semi-discrete setting \eqref{eq:heat_equation1-intro},
continuous in time and discrete in space,
where estimates such as \eqref{eq:intro-Gaussian}
are not available.

For simplicity of
notation in this introduction,
let us consider the one-dimensional
case $d=1$. The fundamental
solution for the semi-discrete heat operator
$$
\frac{\partial}{\partial t}
-\frac{1}{2} c_0 \nabla_+ \nabla_-,
\qquad \nabla_{\pm} a_\alpha
\coloneqq \pm \frac{a_{\alpha \pm 1} - a_\alpha}{\Dx},
$$
with initial data $a_\alpha(0) = \delta_{\alpha}
\coloneqq \Dx^{-1} \one{\alpha = 0}$
and a constant $c_0 > 0$, is given by
\begin{equation}\label{eq:disc-heat}
	a_\alpha(t) = \frac{1}{\Dx} e^{-\frac{2c_0t}{\Dx^2}}
	I_\alpha\left(\frac{2c_0t}{\Dx^2}\right),
	\quad t \ge 0, \,\, \alpha \in \Z,
\end{equation}
where $I_k(t)$ is the Bessel function of the first kind and
order $k \in \Z$, see, e.g., \cite{BW2016,Feller:1966aa,Fjordholm:2023aa}.
The function $a_\alpha(t)$ represents the semi-discrete
heat kernel on $\Dx \Z$, playing a role analogous to
the Gaussian \eqref{eq:heat-kernel-cont} on $\R$. It serves
as the probability density function for a continuous-time
symmetric random walk on $\Dx \Z$, with $c_0$ denoting
the intensity of transitions between neighbouring grid points.

Given the semi-discrete heat kernel $a_\alpha(t)$, the
solution of the Cauchy problem with initial data
$\psi_\alpha$ is given by a discrete convolution between
$a_\alpha(t)$ and $\psi_\alpha(t)$.
Similarly, if the equation includes an inhomogeneity $f$,
one has a Duhamel solution formula for the corresponding
Cauchy problem. For more details, see
Lemma \ref{lem:duhamel_representation}.

During the 1980s and 1990s, remarkable efforts
were made to achieve precise
Gaussian bounds for the heat kernel associated with second-order
elliptic operators on domains in $\R^d$ and on manifolds.
This period saw the development of various deep techniques
to establish these estimates, as documented in the books
\cite{Davies:1989ab,Grigoryan:2009ab,Varopoulos:1992aa}.
The general methodologies were not confined to continuous settings
but were also adapted for discrete structures, particularly
to derive heat kernel bounds for continuous-time
random walks on discrete graphs (see, e.g., \cite{Pang:1993aa})
and symmetric simple exclusion processes 
(see, e.g., \cite{Lan2005}).
The estimates from references \cite{Davies:1989ab,Pang:1993aa}
do not fully cover the regimes relevant to our study, 
so we will not rely on them in this paper. 
However, their relation to our work will be discussed in 
Section \ref{sec:daviespang}.
We also mention the recent papers
\cite{CJKS2023,Jorgenson:2024aa}, which use a
parametrix approach to derive formulas for the
heat kernel on a graph $G$. For example, when $G$
(e.g., $\N_0$) 
is a subgraph of a 
larger graph $\tilde{G}$ (e.g., $\Z$), they derive the
heat kernel on $G$ from the heat kernel on
$\tilde{G}$ (e.g., \eqref{eq:disc-heat}
with $\Delta x = 1$) restricted to $G$.

Closer in spirit to our paper is the recent interest
in deriving estimates for {\em constant} coefficient
semi-discrete heat equations by working directly
with formulas like \eqref{eq:disc-heat}.
These studies \cite{AD2022,AGMP2021,ADM2024,CGRTV2017,Ignat:2006aa,
Lizama:2024ab,Slavik:2020aa,Slavik:2022aa}, employing
detailed estimates involving Bessel functions, have
explored the long-time behavior of solutions and
addressed various questions related
to the fundamental solution. This includes asymptotic
pointwise and $\ell^p$ decay results,
characterisations of discrete H\"older and Besov
spaces using ``semi-discrete" heat/Poisson semigroups,
and investigations in ``discrete" harmonic analysis such
as the $\ell^p$ boundedness of Riesz transforms and
Littlewood--Paley square functions.
In this paper, we confine our analysis of (1.1) to the case 
where the coefficient matrix $c_\alpha$ is diagonal, so that the
spatial directions are decoupled. One motivation for this 
restriction is that our approach relies on certain one-dimensional 
constant-coefficient estimates established in the 
previously discussed works. Generalizing these estimates to 
higher dimensions appears to be non-trivial 
\cite{AD2022,AGMP2021,ADM2024}.

Unlike the previously cited papers, we focus
on fundamental solutions and $\Dx$-independent
pointwise/$L^p$ estimates for the {\em variable coefficient} case.
Our approach will involve setting up a parametrix based on the
constant coefficient heat kernel \eqref{eq:disc-heat}.
Gaussian estimates for \eqref{eq:disc-heat} are unavailable near $t = 0$.
This makes the parametrix approach for the semi-discrete
problem more challenging than the
fully discrete or fully continuous cases.
The fully discrete setting, involving discretisation in
both time and space, corresponds to well-established numerical
methods, where the relevant estimates
are classical (see also \cite{AGR2023}).

In the absence of Gaussian estimates
for the constant coefficient heat kernel \eqref{eq:disc-heat},
we derive ``Lorentzian" estimates tailored to different temporal
regimes: $\{t \le C \Delta x^2\}$ and its complement.
For $m \ge 0$, let $(\nabla_\pm)^m$ denote $m$
successive applications of the discrete
derivative operators $\nabla_\pm$.
We establish the following pointwise estimates
for \eqref{eq:disc-heat} (see Proposition \ref{lem:green_estimates})
for $m=0,1,2$:
\begin{equation}\label{eq:pointwise_da-intro}
	\begin{aligned}
		\abs{(\nabla_\pm)^m a_\alpha(t)}
		\lesssim
		\bk{\frac1{\sqrt{2\overline{c}}}
		\wedge\frac{t^{\hf}}{\Dx}}^{Z(\alpha)}
		t^{-\bk{1 + m}/2}
		\bk{1 + \frac{\abs{x_\alpha}^2}{2\overline{c}t}
		+\frac{\abs{x_\alpha}^3}{\abs{2\overline{c}t}^{3/2}} }^{-1},
 	\end{aligned}
\end{equation}
for $t>0$, $\alpha\in \Z$, $x_\alpha=\alpha \Dx$,
and $Z(\alpha) = \one{\alpha = 0}$, and $\overline{c} \coloneqq \sup_{i \le d} c^i_0$.
Importantly, we identify a region $\{4 \overline{c} t < \Dx^2\}$
around $(t, \alpha) = (0, 0)$ where the (expected)
negative exponent factor $t^{-\bk{1 + m}/2}$ is
partly offset by the $t^{\hf}$ term. In particular,
$|a_0(t)|$ is bounded uniformly in $t$ and in $\Dx$ down to $t = 0$.
This insight is crucial for the convergence of
the parametrix method. The estimates \eqref{eq:pointwise_da-intro}
are sharp or near optimal for small times.
However, they are not asymptotically
accurate, and we also show that it is possible to
obtain pointwise Gaussian estimates
for larger times $t \ge C\Dx \abs{x_\alpha}$
(see Proposition \ref{thm:a_pointwise_exponential}
and \cite{Pang:1993aa}).

The Lorentzian estimates \eqref{eq:pointwise_da-intro}
of \eqref{eq:disc-heat} stand in as substitutes
of Gaussian estimates \eqref{eq:intro-Gaussian} for us.
This allows us to demonstrate that the
semi-discrete parametrix approach is well-defined
and converges to the fundamental
solution $\Gamma_{\alpha,\beta}(t)$
of the variable coefficient heat equation, i.e., the $C^1$
solution of the infinite ODE system
\begin{equation}\label{eq:fund_soln_variable_heat-intro}
	\frac{\d}{\d t}  \Gamma_{\alpha, \beta}(t)
	=  c_\alpha
	\nabla_+ \nabla_- \Gamma_{\alpha, \beta}(t),
	\quad
	 {\Gamma}_{\alpha,\beta}(0)
	= \delta_{\alpha-\beta},
	\quad
	t>0, \,\, \alpha,\beta\in\Z,
\end{equation}
where $\nabla_+ \nabla_-$ acts in the $\alpha$-slot.
Specifically, we prove that an iteratively corrected Neumann series
involving $n$-fold convolutions---spatially discrete
versions of \eqref{eq:intro-parametrix} and \eqref{eq:intro-Phi}
based on \eqref{eq:disc-heat}---converges to the fundamental solution
$\Gamma_{\alpha,\beta}$.

We rely on the $\thf$ order term in \eqref{eq:pointwise_da-intro}
to absorb certain ``bad" factors, ensuring the
convergence of the series and achieving the
following $\Dx$-independent pointwise
decay estimates (see Theorem \ref{lem:full_greensfunction}):
\begin{equation}\label{eq:Gamma-ab-ndiff-est1-intro}
	\begin{aligned}
		\bigl|(\nabla_\pm)^m \Gamma_{\alpha, \beta}(t)\bigr|
		&\lesssim
		\bk{\frac1{\sqrt{2\overline{c}}}
		\wedge\frac{t^{\hf}}{\Dx}}^{Z(\alpha - \beta)}
		t^{-\bk{1 + m}/2}
		\bk{1+\frac{{\big|x_\alpha -x_\beta\big|}^2}
		{2\overline{c}t}}^{-1},
	\end{aligned}
\end{equation}
for $0<t<T$, $\alpha\in \Z$,  $m = 0,1,\ldots$, and
$\overline{c} =\sup_{\alpha\in \Z} c_\alpha$.
We derive \eqref{eq:Gamma-ab-ndiff-est1-intro} using the
Lorentzian heat kernel estimates \eqref{eq:pointwise_da-intro}
and various bounds on convolution products
of Cauchy probability densities. Recall that
a Cauchy density is given by
\begin{equation}\label{eq:Cauchy-density}
	\mathcal{L}(x; x_0,\gamma)=\frac{\abs{\gamma}}{\pi}
	\frac{1}{(x-x_0)^2+\gamma^2},
	\quad x \in \R,
\end{equation}
where $x_0$ is the location parameter and $\abs{\gamma}$
is the scale parameter. Due to the
heavy tails of \eqref{eq:Cauchy-density}, random
variables with the Cauchy distribution,
also known as the Lorentz distribution, do
not have finite mean or variance. However,
the Cauchy distribution is closed
under convolution \cite[page 51]{Feller:1966aa},
implying that the sum of Cauchy random variables
also follows a Cauchy distribution.

Using the variable coefficient fundamental solution
\eqref{eq:fund_soln_variable_heat-intro} and
the pointwise decay estimates in
\eqref{eq:Gamma-ab-ndiff-est1-intro}, which
also imply $\ell^p$ estimates, we can formulate a Duhamel
solution formula for \eqref{eq:heat_equation1-intro} and
derive corresponding solution estimates.
For details, see Section \ref{sec:variable_coeff}.

Our work is inspired by \cite{Fjordholm:2023aa} 
and ongoing research in the numerical analysis 
of stochastic transport equations.
In \cite{Fjordholm:2023aa}, we identified
a regularisation by noise mechanism
in numerical methods that requires $\ell^p$ estimates
of $\Gamma_{\alpha,\beta}$. We
established these estimates directly, avoiding the need for
pointwise estimates and the convergence of the
parametrix approach over $[0,T]$ for any $T>0$,
which is the main focus of the present paper. 
We believe it is of independent interest to provide
a comprehensive analysis of the parametrix approach for
constructing solutions and establishing grid size independent estimates for 
semi-discrete heat equations with variable coefficients.
Since regularisation by noise is often 
achieved through the Itô-Tanaka trick (sometimes 
requiring additional Malliavin differentiability 
assumptions), the numerical observation of such 
regularisation naturally involves solving a parabolic equation. 
Beyond this motivation, parabolic equations 
play a central role in a wide range of models. 
The parametrix method is a fundamental tool 
in the analysis of linear parabolic equations 
with variable coefficients. Developing a 
numerical parametrix method for semi-discrete 
second-order parabolic equations is therefore 
of intrinsic interest and may also find applications 
in other contexts.

The remaining part of this paper is organised as follows:
We first establish some notational groundwork
in Section \ref{sec:discrete_setting}.
In Section \ref{sec:greens_functions}, we derive
refined pointwise estimates for the constant
coefficient semi-discrete heat equation.
These estimates are then used in
Section \ref{sec:variable_coeff} to
construct a parametrix for variable coefficient equations.
Finally, Section \ref{sec:aux_calc} presents the main
induction argument that establishes the pointwise
bounds on the constructed parametrix.

\section{Discrete notations}\label{sec:discrete_setting}
We begin by introducing notations for finite
difference operators, discrete spaces,
and stating a discrete convolution inequality.

Let $\Dx>0$ be the grid size parameter.
For indices $\alpha\in\Z^d$
we denote by $x_\alpha\coloneqq \alpha\Dx$
the grid points and we partition the
spatial domain $\R^d$ into grid cells
$$
\cell_\alpha \coloneqq x_\alpha
+[-\nicefrac{\Dx}{2}, \nicefrac{\Dx}{2})^d.
$$
We will denote by $L^\infty_\Dx(\R^d)$ 
the collection of functions that are uniformly bounded and 
piecewise constant on each cell $\cell_\alpha$,
and we set
$$
L^p_\Dx(\R^d)\coloneqq
L^\infty_\Dx(\R^d)\cap L^p(\R^d)
$$
for any $p\in[1,\infty]$. To distinguish
elements of these discrete spaces from
their continuous counterparts, we will
denote elements of $L^p_\Dx(\R^d)$
by bold symbols. By definition then,
any function $\bm{f}\in L^\infty_\Dx(\R^d)$
can be written as
\begin{equation*}
	\bm{f}(x) = \sum_{\alpha\in\Z^d} f_\alpha
	\one{\cell_\alpha}(x), \quad x\in\R^d, \,\, f_\alpha \in \R, 
\end{equation*}
and
$$
\|\bm{f}\|_{L^p(\R^d)} = \begin{cases}
\Bigl(\sum_{\alpha\in\Z^d}
|f_\alpha|^p \Dx^d\Bigr)^{1/p} & p<\infty,\\
\sup_{\alpha\in\Z^d}|f_\alpha| & p=\infty.
\end{cases}
$$
We will interchangeably
view elements $\bm{f}$ of $L^p_\Dx(\R^d)$ either
as functions of $x\in\R^d$ or as sequences
$f=f_\alpha=\{f_\alpha\}_{\alpha\in \Z^d}$
parametrised by $\alpha\in\Z^d$.
In the latter scenario, the aforementioned norm
is just the $\ell^p(\Dx^d)$-norm, which
scales the conventional discrete $\ell^p$-norm
by the grid volume $\Dx^d$.

For clarity and simplicity in notation, instead
of $\|\bm{f}\|_{L^p(\R^d)}$, we often use
$\|\bm{f}\|_{L^p_\alpha}$. In the case of mixed norms
for functions depending on two variables, denoted by
$f=f_{\alpha,\beta}=
\{f_{\alpha,\beta}\}_{\alpha,\beta \in \Z^d}$,
which can also be represented by a piecewise constant
function $\bm{f}$, we will employ the notation
$\|\bm{f}\|_{L^{p_1}_\alpha L^{p_2}_\beta}$.
This notation indicates that we first calculate
the $\ell^{p_2}(\Dx^d)$ norm with respect to the
$\beta$ index, followed by the $\ell^{p_1}(\Dx^d)$ norm
with respect to the $\alpha$ index.
When dealing with continuous time $t$,
similar notations will be employed. For instance,
if $f=f_\alpha(t)$, we will use mixed norms such
as $\|f\|_{L^q_t L^p_\alpha}$. This notation indicates that
we first calculate the $\ell^p(\Dx^d)$ norm over the spatial
indices $\alpha$, followed by computing the continuous
$L^p$ norm of the resulting function with
respect to time $t$.

Vector valued versions $L^p_\Dx(\R^d,\R^d)$
of the aforementioned spaces are defined
similarly; the components of
vectors in $\R^d$ will be indexed
by a superscript ${}^i$ for $i=1,\dots,d$.
We let $e_1,\dots,e_d\in\R^d$ denote the
canonical basis, $e_j^i = \delta_{i,j}$.

For $\bm{f}\in {L^\infty_\Dx}(\R^d)$ we define the
(vector-valued) difference operators $\nabla_+$, 
$\nabla_-$, defined component-wise as
\begin{equation}\label{eq:difference_operators}
\begin{gathered}
	\nabla_+^j f_\alpha \coloneqq
	\frac{f_{\alpha + e_j} - f_{\alpha}}{\Delta x},
	\qquad
	\nabla_-^j f_\alpha \coloneqq
	\frac{f_{\alpha} - f_{\alpha-e_j}}{\Delta x}.
\end{gathered}
\end{equation}
for $\alpha\in\Z^d$ and $j=1,\dots,d$. Throughout this
paper we will write $(\nabla_+ f)_\alpha
= \nabla_+ f_\alpha$,
$(\nabla_- f)_\alpha = \nabla_- f_\alpha$.

For discrete functions of two variables
$F=F_{\alpha,\beta}$, $G=G_{\alpha,\beta}$ we
define the discrete convolution
\begin{equation*}
	\bigl(\bm{F}\dconv\bm{G}\bigr)_{\alpha,\beta}
	\coloneqq \sum_{\eta\in\Z^d}
	F_{\alpha,\eta}G_{\eta,\beta}\,\Dx^d.
\end{equation*}

The more familiar variants for functions
of one spatial variable, $f=f_\alpha$ or $f=f_\alpha(t)$,
are obtained by setting
$F_{\alpha,\beta} =f_{\alpha-\beta}$
and $G_{\alpha,\beta}=g_\alpha$, leading to
\begin{equation*}
	(\bm{f}\dconv\bm{g})_\alpha
	= \sum_{\eta\in\Z^d} f_{\alpha-\eta}g_\eta\,\Dx^d.
\end{equation*}
We have the following variant of
Young's convolution inequality
\cite[Lemma 3.1]{Fjordholm:2023aa}.

\begin{lemma}\label{lem:young}
Let $p_1,p_2,q_1,q_2\in[1,\infty]$ satisfy
$\frac{1}{p_2}+\frac{1}{q_1} = 1$.
If $\bm{F}\in L^{p_1}_\alpha L^{p_2}_\beta$
and $\bm{G}\in L^{q_1}_\alpha L^{q_2}_\beta$ then
$$
\|\bm{F}\dconv\bm{G}\|_{L^{p_1}_\alpha L^{q_2}_\beta}
\leq \|\bm{F}\|_{L^{p_1}_\alpha L^{p_2}_\beta}
\|\bm{G}\|_{L^{q_1}_\alpha L^{q_2}_\beta}.
$$
\end{lemma}

\section{Constant coefficients}
\label{sec:greens_functions}

In this section, we will consider the following semi-discrete
(anisotropic) heat equation with constant coefficients
$c^1,\ldots,c^d>0$:
\begin{equation}\label{eq:discrete_heat_const}
	\frac{\d}{\d t}  u_\alpha
	= \sum_{j=1}^d c^j \nabla^j_- \nabla^j_+  u_\alpha.
\end{equation}
We will provide refined a priori estimates for
the fundamental solution of \eqref{eq:discrete_heat_const}.
Additionally, we will present
the solution formula for the nonhomogeneous version
of \eqref{eq:discrete_heat_const} with a source
vector $\bm{f}$, along with $L^\infty$ estimates for
the solution and its (discrete) derivative.

The solution to the semi-discrete heat equation
\eqref{eq:discrete_heat_const} can be expressed
as the convolution $\bm{u}(t)=\bm{u}(0)\dconv\bm{a}(t)$, where
$\bm{a}$ is the fundamental solution.
Letting $\bm{\delta}\in L^1_\Dx(\R^d)$
denote the discrete Dirac measure:
$$
\delta_\alpha \coloneqq
\Dx^{-d}\one{\alpha=0}
=
\begin{cases}
	\Dx^{-d} & \text{if $\alpha=0$},
	\\
	0 & \text{if $\alpha\neq 0$},
\end{cases}
$$
we define $\bm{a}(t)=\{a_\alpha(t)\}_{\alpha\in \Z^d}$
as the solution to the ODE system
\begin{align}\label{eq:discheateq}
	\frac{\d }{\d t} a_\alpha(t)
	= \nabla_+ \cdot \bigl(c\nabla_- a_\alpha(t)\bigr),
	\qquad a_\alpha(0) = {\delta}_\alpha,
	\qquad t>0, \,\, \alpha\in\Z^d,
\end{align}
where $c=\diag(c^1,\ldots,c^d)$.
One formal representation of $a_\alpha(t)$
can be given by
\begin{equation}\label{eq:greensfunction}
	a_\alpha(t) = \sum_{i=0}^\infty
	\frac{t^i}{i!} \bigl(\nabla_+
	\cdot c\nabla_-\bigr)^i{\delta}_\alpha.
\end{equation}

Before proceeding, and for future reference, let us
recall that there is an explicit
solution formula for the following
one-dimensional problem ($d=1$, $c^1=1$, $\Dx=1$):
\begin{equation}\label{eq:sd-heat-1D}
	\frac{\d}{\d t}  a_n(t)
	=a_{n-1}-2a_n+a_{n+1},
	\quad
	a_n(0)= \one{n = 0},
	\quad t>0,\,\, n\in \Z.
\end{equation}
The solution formula is
\begin{equation}\label{eq:sol-1D}
	a_n(t) = e^{-2t} I_n(2t),
\end{equation}
where $I_n$ is the modified Bessel function of
the first kind and of order $n$. For $n\in \Z$
and $r\ge 0$, $I_n(r)$ takes the form
(see, e.g., \cite[(9.5.2)]{BW2016})
\begin{equation}\label{eq:Bessel-1D-tmp1}
	I_n(r)=\sum_{k=0}^\infty \frac{1}{k!(k+n)!}
	\left(\frac{r}{2}\right)^{2k+n},
	\quad
	n\in \N_0,
	\qquad
	I_{-n}(r)= I_n(r),
	\quad n\in \N.
\end{equation}
This function satisfies $I_n(0)= \one{n=0}$.
Alternatively, $I_n(r)$ can be
represented by the integral
\begin{equation*}
	I_n(r)=\frac{1}{\pi}\int_0^\pi
	e^{r\cos\theta}\cos(n\theta)\,d\theta.
\end{equation*}
This solution formula \eqref{eq:sol-1D}
is widely known. For relevant
references and a list of important properties
that can be derived from it, see the recent
papers \cite{AD2022,AGMP2021,ADM2024,CGRTV2017,Ignat:2006aa}.
We will recall some of these properties for later use.

The solution operators
$\bar{a}=\{\bar{a}_n\}_{n\in \Z}
\mapsto \mathcal{S}(t)\bar{a}(n)
\coloneqq\sum_{m\in\mathbb{Z}} a_{n-m}(t)\, \bar{a}_m$ forms
a strongly continuous one-parameter semigroup.
These operators are contractive with respect to the
$\ell^p$ norms (including $\ell^1$, $\ell^2$, and $\ell^\infty$),
preserve nonnegativity, and leave the constant
function unchanged (the Markov property).
This structure ensures a unique solution to the
semi-discrete heat equation on $\Z$ with initial
data $\bar{a}\in \ell^\infty$, with continuous time
dependence and stability
relative to the initial data.
For details, see \cite{CGRTV2017} (and
Proposition \ref{lem:green_estimates} below).

\subsection{The constant coefficient fundamental solution near $t = 0$}

We establish new bounds for the 
fundamental solution of the constant-coefficient 
equation \eqref{eq:discrete_heat_const}. A key 
feature of these bounds is their stability under 
spatio-temporal convolutions, which makes them 
well suited for applications based on a coefficient-freezing 
strategy. In particular, they will be used in 
subsequent sections to construct a parametrix 
and to derive pointwise estimates in Section 
\ref{sec:aux_calc} for the fundamental solution 
of the variable-coefficient equation 
\eqref{eq:heat_variable_coeff}. The analysis is 
based on the representation of the fundamental 
solution in terms of modified Bessel functions  
\cite{Hoff:1985zm}, together with the asymptotic 
results proved in \cite{AD2022,AGMP2021,CGRTV2017}.

Let $R=\frac{|n|^2}{t}$ with $n \in \Z$ and $t > 0$.
Below, we distinguish between two cases: $R\leq 1$ and
$R\geq 1$. As a convention, for $n=0$, $R\leq 1$
is interpreted as $t>0$, while the case $R\geq 1$
is considered void. Then there exists a constant $C > 0$
(independent of $n$ and $t$) such that
\begin{equation} \label{eq:a_n_bound}
	|a_n(t)| \leq
	\begin{cases}
		\frac{C}{t^{1/2}} & \text{if } R \leq 1, \\
		\frac{C t}{|n|^3} & \text{if } R \geq 1.
	\end{cases}
\end{equation}
If $\nabla^+ a_n(t)
=a_{n+1}(t)-a_n(t)$, then
\begin{equation} \label{eq:first_diff}
	|\nabla^+a_n(t)| \leq
	\begin{cases}
		\frac{C}{t^{1/2}} & \text{if } R \leq 1, \\
		\frac{C t}{|n|^4} & \text{if } R \geq 1.
	\end{cases}
\end{equation}
For the second order difference
$\nabla^+\nabla^- a_n(t)
= a_{n-1}(t) - 2a_n(t) + a_{n+1}(t)$,
\begin{equation} \label{eq:second_diff}
	|\nabla^+\nabla^- a_n(t)| \leq
	\begin{cases}
		\frac{C}{t^{3/2}} & \text{if } R \leq 1, \\
		\frac{C}{|n|^3} & \text{if } R \geq 1.
	\end{cases}
\end{equation}
For proofs of these estimates, see
\cite[Lemma 4.1 and Remark 4.2]{AGMP2021}.
Furthermore, in \cite[Lemma 2.3]{AD2022}, it
is proven that, for $m\in \N$,
$n\in \Z$, and $t>0$,
\begin{equation}\label{eq:time-decay-tmp}
	|(\nabla^{+})^m a_n(t)|
	\leq \frac{C_n}{t^{\lfloor (m+1)/2 \rfloor+1/2}}.
\end{equation}
According to \cite[Lemma 2.4]{AD2022},
for $m \in \N$, $n \in \N_0$ (nonnegative),
and $t>0$, we have
\begin{align*}
	|(\nabla^{+})^m a_n(t)|
	\leq&\ \frac{C_m}{t^{m/2}}\sum_{u=0}^{\lfloor m/2 \rfloor}
	\left( \frac{(n+1/2)^2}{t}\right)^{m/2-u}a_{n+m-2u}(t)
	\\
	&
	+C_m a_n(t)\!\!\sum_{u=\lfloor m/2 \rfloor+1}^{m-1}
	\frac{1}{t^u},
\end{align*}
where $C_m$ is a positive constant which
is independent on $t$ and $n$.
For $m=1$, the above estimate reads
\begin{equation}\label{eq:nabla-plus-an-tmp}
	|\nabla^{+}a_n(t)|
	\le \frac{C_1(n+1/2)}{t}a_{n+1}(t).
\end{equation}
while, for $m=2$ and $(\nabla^{+})^2 a_n(t)
=a_n(t)-2a_{n+1}+a_{n+2}(t)$,
\begin{equation}\label{eq:nabla-plus-squared-an-tmp}
	|(\nabla^{+})^2 a_n(t)|
	\le \frac{C_2(n+1/2)^2}{t^{2}}a_{n+2}(t)
	+\frac{C_2}{t}a_n(t).
\end{equation}

We will require the estimates
\eqref{eq:nabla-plus-an-tmp} and
\eqref{eq:nabla-plus-squared-an-tmp}
later, specifically
in the ``large time" regime
$$
C\sqrt{t} \ge n, \quad \text{for some constant $C>0$}.
$$
First, let us consider $\nabla^{+} a_n(t)$ where
$C\sqrt{t}\ge n\ge 1$, in which case
$n+1/2 \lesssim \sqrt{t}$, so that
$\frac{n+1/2}{t} \lesssim \frac{1}{t^{1/2}}$.
Hence, since ${a_n(t)}$ is decreasing in $n$,
\eqref{eq:nabla-plus-an-tmp} implies
\begin{equation}\label{eq:nabla-plus-an}
	|\nabla^+a_n(t)|
	\lesssim \frac{1}{t^{1/2}}a_n(t)
	\qquad \text{for $C\sqrt{t}\ge n \ge 1$}.
\end{equation}
Similarly, let us consider the second order difference
operator $(\nabla^{+})^2 a_n(t)$ in
the same regime. Moreover, because
$1\le n\le C\sqrt{t}$, it follows that
$(n+1/2)^2\lesssim t$. Therefore,
since $a_{n+2}(t)\le a_n(t)$, the
first term in \eqref{eq:nabla-plus-squared-an-tmp}
satisfies
$$
\frac{C_2(n+1/2)^2}{t^{2}}a_{n+2}(t)
\lesssim\frac{1}{t}a_n(t),
$$
and hence,
\begin{equation}\label{eq:nabla-plus-squared-an}
	| (\nabla^{+})^2 a_n(t)|
	\lesssim \frac{1}{t}a_n(t)
	\qquad \text{for $C\sqrt{t}\ge n \ge 1$}.
\end{equation}

For negative values of $n$, given that
$a_{-n}(t) = a_n(t)$, the bounds
remain similar but not identical. Indeed,
for $C\sqrt{t}\ge n \ge 2$, we have
\begin{equation}\label{eq:nabla-neg-an}
	\begin{split}
		|\nabla^+a_{-n}(t)|&=|a_{-n+1}(t)-a_{-n}(t) |
		=|a_{n-1}(t)-a_{n}(t)|
		\\ &
		\overset{\mathclap{\eqref{eq:nabla-plus-an-tmp}}}{\lesssim}
		\frac{1}{\sqrt t}a_{n}(t)
		=\frac{1}{\sqrt t}a_{-n}(t),
	\end{split}
\end{equation}
and, for $C\sqrt{t}\ge n \ge 3$,
\begin{equation}\label{eq:nabla-neg-squared-an}
\begin{split}
	|(\nabla^{+})^2 a_{-n}(t)|
	&= |a_{-n}(t)-2a_{-n+1}(t)+a_{-n+2}(t)| \\
	&=|a_{n}(t)-2a_{n-1}(t)+a_{n-2}(t)|
	\overset{\eqref{eq:nabla-plus-squared-an}}{\lesssim}
	\frac{1}{t}a_{n-2}(t)
	=\frac{1}{t}a_{-n+2}(t).
\end{split}
\end{equation}

For later reference, note that for $n=0$,
\eqref{eq:nabla-plus-an-tmp} yields
\begin{equation}\label{eq:nabla-plus-a0}
	|\nabla^{+}a_0(t)|
	\lesssim \frac{1}{t}a_{1}(t)
	\le \frac{1}{t}a_{0}(t)
	\qquad \text{for }  t>0.
\end{equation}

Let us now consider the one-dimensional semi-discrete heat
equation with a constant diffusion coefficient $c>0$
and a grid parameter $\Dx>0$:
\begin{equation}\label{eq:sd-heat-c-1D}
	\frac{\d}{\d t}  a_n^{(c)}(t)
	=c\frac{a_{n-1}^{(c)}-2a_n^{(c)}+a_{n+1}^{(c)}}{\Dx^2},
	\qquad
	a_n^{(c)}(0)=\Dx^{-1}\one{n = 0}.
\end{equation}
Rewriting the equation in terms of
the new temporal variable $t\mapsto \frac{c}{\Dx^2}t$
yields \eqref{eq:sd-heat-1D}, whose solution is
provided in \eqref{eq:sol-1D}.
Transforming back to the original variable results in the
following solution formula, incorporating the initial data:
\begin{equation}\label{eq:sol-1D-c}
	a_n^{(c)}(t) = \Dx^{-1} e^{-r} I_n(r)
	=\Dx^{-1}a_n(r/2),
\end{equation}
where $r=\frac{2ct}{\Dx^2}$, and $a_n$ is the solution
of \eqref{eq:sol-1D}.
This suggests that the estimates obtained for
\eqref{eq:sd-heat-1D} can be readily applied
to the solution of \eqref{eq:sd-heat-c-1D}.

In the next proposition, we offer our
first refined estimates for the
(fundamental) solution of
the multi-dimensional problem
\eqref{eq:discrete_heat_const}. 
Whilst bearing a passing 
semblance to \cite[Lemma 4.1]{Fjordholm:2023aa}, 
the core estimate \eqref{eq:pointwise_da} is more accurate 
around $t = 0$, and requires a more involved proof.
Of particular significance, we identify a
region near $\alpha^j = 0$ and $t = 0$,
where a temporal factor with negative power would
otherwise mask much better behaviour (see
\eqref{eq:pointwise_da} below).
This observation will be crucial for our
calculations for a pointwise bound in Section \ref{sec:aux_calc}.

\begin{proposition}[Fundamental solution]
\label{lem:green_estimates}
The series representation \eqref{eq:greensfunction}
for the solution $\bm{a}(t)
=\{a_\alpha(t)\}_{\alpha\in \Z^d}$
of \eqref{eq:discheateq}
converges for all $t\in \R$, uniformly in
$t$ for $t\in [0,T]$, and the limit belongs to
$C([0,T];L^1_\Dx(\R^d))$ for any $T>0$.
Moreover,
\begin{enumerate}[label=(\roman*)]
	\item $a_\alpha(\cdot)\in C^1((0,\infty))$
	has the following representation
	for $\alpha\in\Z^d$ and $t>0$:
	\begin{equation}\label{eq:fundamentalsolution}
    	a_\alpha(t) = \frac{1}{\Dx^d}
    	\prod_{j=1}^d e^{-r^j} I_{\alpha^j}(r^j), \qquad
    	r^j \coloneqq\frac{2c^jt}{\Dx^2}.
   \end{equation}
   Here, the function $I_{n}(r)$ is defined
   in \eqref{eq:Bessel-1D-tmp1}
   for any $n\in \Z$ and $r \ge 0$.

	\item\label{lem:green_estimates_positivity}
	$a_{\alpha}(t) > 0$ for all $\alpha\in\Z^d$ and $t>0$,
	$\sum_{\alpha \in \Z^d} a_\alpha(t) \Dx^d = 1$
	for all $t\geq0$, and
	$\|\bm{a}(t)\|_{L^p(\R^d)}
	\leq \|\bm{a}(0)\|_{L^p(\R^d)} < \infty$
	for all $t>0$, $p\in [1,\infty]$.

	\item\label{lem:green_estimates_derivative-bound}
	For $\ell=1,\dots,d$ and $m = 0,1,2$, recall the definition
	of $\nabla_+^\ell$ in \eqref{eq:difference_operators}, and
	let $(\nabla_+^\ell)^m$ denote $m$
	successive applications of $\nabla_+^\ell$.
	For $\alpha\in \Z^d$, define
	$Z(\alpha) = \# \{j \in \{1,\ldots,d\}
	\mid \alpha^j = 0 \}$. Then
	\begin{equation}\label{eq:pointwise_da}
		\begin{aligned}
			\abs{(\nabla_+^\ell)^m a_\alpha(t)}
			\lesssim\ &\biggl(\frac1{\sqrt{2\overline{c}}}
			\wedge \frac{t^{\hf}}{\Dx}\biggr)^{Z(\alpha)} \\
	   		&\times t^{-\bk{d+m}/2}\prod_{j = 1}^d
			\bk{1+\frac{\abs{x_\alpha^j}^2}{2\overline{c}t}
			+\frac{\abs{x_\alpha^j}^3}{\abs{2\overline{c}t}^{3/2}}}^{-1},
		\end{aligned}
	\end{equation}
	where $\overline{c}\coloneqq \max(c^1,\dots,c^d)$.
	Furthermore, for $p\geq 1$,
	\begin{equation}\label{eq:a_lp}
		\bigl\|(\nabla_+^\ell)^m\bm{a}(t)\bigr\|_{L^p (\R^d)}
		\leq C_{\ell,m,p}t^{d/(2p)-(d+m)/2},
	\end{equation}
	noting that $t^{d/(2p)-(d+m)/2}$ can also be expressed
	as $t^{d/(2p')-m/2}$, where $p'$ is
	the dual exponent of $p$.
\end{enumerate}
\end{proposition}

\begin{remark}\label{rem:gaussian_pointwise}
In Proposition \ref{thm:a_pointwise_exponential},
we show that Gaussian pointwise estimates are
valid for $a_\alpha(t)$ when
$t \gtrsim_{c} \max\limits_{j=1,\ldots,d}
\abs{\alpha^j} \Dx^2$. 
\end{remark}

\begin{proof}
Properties (i) and (ii), along with the
convergence of the series \eqref{eq:greensfunction},
were established in \cite[Propositions 1 and 2]{CGRTV2017}
(see also \cite[Lemma 4.1]{Fjordholm:2023aa}).

Here, we establish the bound \eqref{eq:pointwise_da}.
For any $\alpha =\{\alpha^j\}_{j=1}^d\in \Z^d$
and $t>0$, recalling the definition of $r^j$
(see \eqref{eq:fundamentalsolution}), the notation
\begin{equation}\label{eq:a-alphaj-cj}
	a_{\alpha^j}^{(c^j)}(t)
	=\Dx^{-1}a_{\alpha^j}(r^j/2)=
	\Dx^{-1}e^{-r^j} I_n(r^j)
\end{equation}
refers to the solution formula for the one-dimensional
semi-discrete heat equation with constant
diffusion coefficient $c^j$
and grid spacing $\Dx$, see \eqref{eq:sd-heat-c-1D}
and \eqref{eq:sol-1D-c}. Using this, we
decompose the multi-dimensional solution
formula \eqref{eq:fundamentalsolution} as
\begin{equation}
	\label{eq:fund_final_expression}
	a_\alpha(t)= \prod_{j=1}^d a_{\alpha^j}^{(c^j)}(t)
	=\prod_{j=1}^d\Xi(\alpha^j,r^j)\Psi(\alpha^j,r^j),
\end{equation}
where, for any $n\in \Z$ and $r>0$,
\begin{align}
	\Xi(n,r) &\coloneqq
	\frac{1}{\sqrt{r \Dx^2}}
	\bk{1+\frac{n^2}{r}
	+\frac{\abs{n}^3}{\abs{r}^{{3/2}}}}^{-1},
	\label{eq:Xinr-def}
	\\
	\Psi(n,r) & \coloneqq
	\frac{e^{-r}I_n(r)}{\Dx\,\Xi(n,r)}
	= r^{1/2}\bk{1+\frac{n^2}{r}
	+\frac{\abs{n}^3}{\abs{r}^{{3/2}}}}
	e^{-r} I_n(r) \lesssim 1.
	\label{eq:Psinr-def}
\end{align}
We refer to \cite[Lemma 4.1]{Fjordholm:2023aa}
for the uniform bound on $\Psi$, which is
based on \eqref{eq:a_n_bound}.

We will partition $\Z \times [0,\infty)$
into two sets $F$ and $F^c$ to obtain a
sharper estimate on $a_\alpha$
in each direction. Let
\begin{equation}\label{eq:F-def}
	F \coloneqq \{(n,r) \in
	\Z\times [0,\infty): n \neq 0\}
	\cup \{(0,r)\in \Z\times [0,\infty): r \ge 1\},
\end{equation}
so that its complement is
\begin{equation}\label{eq:Fc-def}
	F^{\tt c} =\{(0,r)
	\in \Z \times [0,\infty): r < 1\}.
\end{equation}
A related decomposition appears
in \cite[page 20]{AGMP2021}.
This allows us to split
\eqref{eq:fund_final_expression}
into two parts as follows:
\begin{equation}\label{eq:a-split-tmp1}
	a_\alpha(t)
	=\prod_{j=1}^d
	\Bigl(
	\Xi(\alpha^j,r^j)
	\Psi(\alpha^j,r^j)\one{F}(\alpha^j,r)
	+ a_{\alpha^j}^{(c^j)}(t)
	\one{F^{\tt c}}(\alpha^j,r)
	\Bigr),
\end{equation}
where the variable $r$ within the
indicator functions is defined by
\begin{equation}\label{eq:def-r}
	r\coloneqq\max_{j=1,\ldots,d}r^j
	=2\overline{c} t/\Dx^2, \qquad
	\overline{c} \coloneqq \max_{j=1,\ldots,d} c^j.
\end{equation}

Regarding the second term on the
right-hand side of \eqref{eq:a-split-tmp1},
it is evident that for
$(n,r) \in F^{\tt c} \Longleftrightarrow \left(n=0, \, \,
\sqrt{2 \overline{c} t}<\Dx\right)$, the inequality
\[
a_0^{(c^j)}(r^j/2)=\Dx^{-1}e^{-r^j} I_0(r^j) \lesssim \Xi(0,r^j)
= \frac1{\sqrt{r^j \Dx^2}}=\frac{1}{\sqrt{2c^jt}}
\]
(which follows from \eqref{eq:Psinr-def})
becomes less significant when compared
to the contraction estimate
$a_0^{(c^j)}(r^j)\le a_0^{(c^j)}(0)
\le \frac{1}{\Dx}$, recognising that
$$
\frac{1}{\sqrt{2c^jt}}
\ge \frac{1}{\sqrt{2\overline{c}t}}
>\frac{1}{\Dx}.
$$
Whilst a uniform $\frac{1}{\sqrt{t}}$ bound
at small times holds for any $n$, it
is necessary to ensure decay as
$\abs{n} \to \infty$, even when $t$ is close to $0$
This is what we proceed with next,
keeping in mind that Gaussian estimates are not
available in the semi-discrete setting for small times.

From the previous discussion and
\eqref{eq:a-split-tmp1},
we see there exists a constant
$C >0$ depending only on $d$
and $c=\diag(c^1,\ldots,c^d)$ such that
\begin{equation}\label{eq:a_F_bound1}
	\begin{aligned}
 		a_\alpha(t)
		\le C \prod_{j = 1}^d
		\bigg( \Xi(\alpha^j,r)
		\one{F}(\alpha^j,r)
		+\frac1\Dx \one{F^{\tt c}}(\alpha^j,r)\bigg),
	\end{aligned}
\end{equation}
where we have also used
that $\Xi(\alpha^j,r^j)\lesssim
\Xi(\alpha^j,r)$, recalling that
$r^j$ and $r$ are defined in
\eqref{eq:fundamentalsolution}
and \eqref{eq:def-r}, respectively.
The product can be expanded as:
\begin{align}\label{eq:a_F_bound2}
	\sum_{I \subseteq \{1,\ldots, d\}}
	\Biggl(\, \prod_{j \in I} \Xi(\alpha^j,r)
	\one{F}(\alpha^j,r)
	\prod_{j \not\in I}\frac1\Dx
	\one{F^{\tt c}}(\alpha^j,r)\Biggr).
\end{align}
Going forward, we will express
\begin{equation}\label{eq:Xi-def-tmp0}
	\Xi(\alpha^j,r)=\frac{1}{\sqrt{2\overline{c}t}}
	\bk{1+\frac{\abs{x_\alpha^j}^2}{2 \overline{c}t}
	+\frac{\abs{x_\alpha^j}^3}{\abs{2 \overline{c}t}^{3/2}}}^{-1}
\end{equation}
in terms of $t$ and $x_\alpha=\alpha \Dx$ instead of $r$
and $n=\alpha^j$ (see \eqref{eq:Xinr-def}).

Because of the indicator on $F^{\tt c}$,
unless $I=\{1,\ldots,d\}$, the summand
indexed by $I$ in \eqref{eq:a_F_bound2} is zero
for $2 \overline{c} t \ge \Dx^2$.
Using this and \eqref{eq:Xinr-def},
\eqref{eq:def-r} yields
\begin{equation}\label{eq:a-alpha-large-t}
	\abs{a_\alpha(t)}
	\one{\{2 \overline{c} t \ge \Dx^2\}}
	\lesssim
	t^{-d/2} \prod_{j = 1}^d
	\bk{1 + \frac{\abs{x_\alpha^j}^2}{2 \overline{c}t}
	+\frac{\abs{x_\alpha^j}^3}{\abs{2 \overline{c}t}^{3/2}}}^{-1}
	\one{\{2 \overline{c} t \ge \Dx^2\}},
\end{equation}
noting that $1\lesssim \frac{t^{1/2}}{\Dx}$ in this case.
On the other hand, when $2 \overline{c} t <  \Dx^2$ we have
$(0,r^j) \notin F$ for all $j$, and hence
each summand \eqref{eq:a_F_bound2} is zero
unless $I$ is exactly the set of indices
$j$ for which $\alpha^j \neq 0$.
Therefore, with $Z(\alpha)$ being the
number of entries of $\alpha$ which are zero,
\begin{align}
	\abs{a_\alpha(t)}
	\one{\{2\overline{c}t<\Dx^2\}}
	&\lesssim \Dx^{-Z(\alpha)}
	\prod_{\{j:\alpha^j \neq 0\}} t^{-\hf}
	\bk{1+\frac{\abs{x_\alpha^j}^2}{2 \overline{c}t}
	+ \frac{\abs{x_\alpha^j}^3}{\abs{2 \overline{c}t}^{3/2}} }^{-1}
	\one{\{2 \overline{c} t <  \Dx^2\}}	\notag \\
&= \Dx^{-Z(\alpha)} t^{-(d-Z(\alpha))/2}
	\prod_{j=1}^d\bk{1+\frac{\abs{x_\alpha^j}^2}{2 \overline{c}t}
	+ \frac{\abs{x_\alpha^j}^3}{\abs{2 \overline{c}t}^{3/2}} }^{-1}
	\one{\{2 \overline{c} t <  \Dx^2\}},
	\label{eq:a-alpha-small-t}
\end{align}
where $\Dx^{-Z(\alpha)} t^{-(d-Z(\alpha))/2}=
\left(\frac{t^{1/2}}{\Dx}\right)^{Z(\alpha)}t^{-d/2}
\le \left(\frac{1}{\sqrt{2 \overline{c}}}
\right)^{Z(\alpha)}t^{-d/2}$.
Adding \eqref{eq:a-alpha-large-t}
and \eqref{eq:a-alpha-small-t}
results in \eqref{eq:pointwise_da} for $m=0$.

Let us now turn our attention to \eqref{eq:a_lp} ($m=0$).
From \eqref{eq:pointwise_da} ($m=0$),
for $\alpha=(\alpha^1,\ldots,\alpha^d)\in \Z^d$ and $t>0$,
\begin{align*}
	\abs{a_\alpha(t)}^p
	\lesssim t^{-dp/2} \prod_{j=1}^d
	\bk{1+\left(\frac{|\alpha^j|}{\sqrt{r}}\right)^2
	+ \left(\frac{|\alpha^j|}{\sqrt{r}}\right)^3}^{-p},
\end{align*}
considering that $x_\alpha^j=\alpha^j\Dx$
and $r=2\overline{c}t/\Dx^2$ (see \eqref{eq:def-r}).
By symmetry and the integral comparison test,
\begin{align}
	\norm{{\bm a}(t)}_{L^p\alpha}^p
	&\lesssim
	\Dx^d t^{-dp/2}
	\left(\int_0^\infty\bk{1+\left(\frac{x}{\sqrt{r}}\right)^2
	+ \left(\frac{x}{\sqrt{r}}\right)^3}^{-p}\, dx\right)^d
	\notag \\ & \lesssim
	\Dx^d t^{-dp/2} r^{d/2}
	\left(\,\int_0^\infty\bk{1+y^2+y^3}^{-p}\, dy\right)^d
	\lesssim t^{-dp/2+d/2},
	\label{eq:comp-int-test}
\end{align}
since the $y$-integral is clearly convergent.
Taking the $p$th root of both sides, we obtain
the estimate $\norm{a(t)}_{L^p(\R^d)}
\lesssim t^{d/(2p)-d/2}$, which confirms
\eqref{eq:a_lp} for $m=0$.

Applying the $m$th order difference
in the $\ell$th direction to
\eqref{eq:fund_final_expression} results in
\begin{align}\label{eq:m-deriv}
	(\nabla_+^\ell)^m a_\alpha(t)
	& =
	\Biggl(\prod_{j\neq \ell} a_{\alpha^j}^{(c^j)}(t)\Biggr)
	\underbrace{\left(\frac{1}{\Dx^m}\sum_{k=0}^m
	(-1)^k \binom{m}{k} a_{\alpha^\ell+k}^{(c^\ell)}(t)
	\right)}_{\eqqcolon\mathcal{D}_{\ell,m}(\alpha^\ell,t)},
\end{align}
where, in particular (recalling the definition
of the scaled fundamental solution
$a^{(c^\ell)}_{\alpha^\ell}$ of \eqref{eq:a-alphaj-cj}),
\begin{alignat*}{2}
	\mathcal{D}_{\ell,1}(\alpha^\ell,t)
	&=\left(a_{\alpha^\ell}^{(c^\ell)}(t)
	-a_{\alpha^\ell+1}^{(c^\ell)}(t)\right)/\Dx
	&& \text{($m=1$)},	\\
	\mathcal{D}_{\ell,2}(\alpha^\ell,t)
	&=\left(a_{\alpha^\ell}^{(c^\ell)}(t)
	-2a_{\alpha^\ell+1}^{(c^\ell)}(t)
	+a_{\alpha^\ell+2}^{(c^\ell)}(t)
	\right)/\Dx^2
	&\qquad& \text{($m=2$)}.
\end{alignat*}

Suppose $\alpha^\ell\neq 0$.
In the small time regime
$\sqrt{2\overline{c}t}\le |x_\alpha^\ell|$,
\begin{align}
	\left(1+\frac{|x_\alpha^\ell|^2}{4\overline{c}t}
	+\frac{|x_\alpha^\ell|^3}{(2\overline{c} t)^{3/2}}\right)^{-1}
	& =\left(
	1+\Bigl(\frac{|x_\alpha^\ell|}{\sqrt{2\overline{c}t}}\Bigr)^2
	+\Bigl(\frac{|x_\alpha^\ell|}{\sqrt{2\overline{c} t}}\Bigr)^3
	\right)^{-1}
	\sim \Bigl(\frac{\sqrt{2\overline{c} t}}{|x_\alpha^\ell|}\Bigr)^3,
	\label{eq:small-time-tmp1}
\end{align}
where $a \sim b$ means that both $a \lesssim b$
and $a \gtrsim b$ hold simultaneously.
Recalling \eqref{eq:Xi-def-tmp0} and \eqref{eq:def-r},
we see from \eqref{eq:small-time-tmp1} that
\begin{equation}\label{eq:Xi-lower-bound}
	\frac{(2\overline{c} t)^{3/2}}{|x_\alpha^\ell|^3}
	\lesssim t^{1/2}\, \Xi(\alpha^\ell,r),
	\qquad \text{for $\sqrt{2\overline{c}t}
	\le |x_\alpha^\ell|$}.
\end{equation}
Hence, using \eqref{eq:first_diff},
alongside \eqref{eq:a-alphaj-cj}, \eqref{eq:sd-heat-c-1D},
and \eqref{eq:sol-1D-c}, and considering
that $\sqrt{2\overline{c}t}\le |x_\alpha^\ell|$
($\alpha^\ell\neq 0$)
$\Longrightarrow \frac{|\alpha^\ell|^2}{r^\ell}\ge 1$
and
$\frac{1}{|x_{\alpha}^\ell|} \le (2\overline{c} t)^{-1/2}$,
we derive
\begin{align*}
	\abs{\mathcal{D}_{\ell,1}(\alpha^\ell,t)}
	& =\Dx^{-2}\abs{a_{\alpha^\ell+1}(r^\ell)
	-a_{\alpha^\ell}(r^\ell)}
	\lesssim \Dx^{-2}
	\frac{r^\ell}{|\alpha^\ell|^4}
	\le \frac{2\overline{c}t}{|x_{\alpha}^\ell|^4}
	\\ & \lesssim
	\frac{(2\overline{c}t)^{1/2}}{|x_{\alpha}^\ell|^3}
	= (2\overline{c}t)^{-1}
	\frac{(2\overline{c}t)^{3/2}}{|x_{\alpha}^\ell|^3}
	\overset{\eqref{eq:Xi-lower-bound}}{\lesssim}
	t^{-1/2}\, \Xi(\alpha^\ell,r),
\end{align*}
and by \eqref{eq:second_diff},
\begin{align*}
	\abs{\mathcal{D}_{\ell,2}(\alpha^\ell,t)}
	& =\Dx^{-3}\abs{a_{\alpha^\ell}(r^\ell)
	-2a_{\alpha^\ell+1}(r^\ell)
	+a_{\alpha^\ell+2}(r^\ell)}
	\\ &
	\lesssim \frac{1}{|x_\alpha^\ell|^3}
	= (2\overline{c}t)^{-3/2}
	\frac{(2\overline{c}t)^{3/2}}{|x_\alpha^\ell|^3}
	\overset{\eqref{eq:Xi-lower-bound}}{\lesssim}
	t^{-1}\, \Xi(\alpha^\ell,r).
\end{align*}
To summarize, within the small time regime where
$\sqrt{2\overline{c}t} \le |x_\alpha^\ell|$,
the following estimates hold:
\begin{equation}\label{eq:small-time-D-est}
	\begin{alignedat}{2}
		\abs{\mathcal{D}_{\ell,1}(\alpha^\ell,t)}
		&\lesssim t^{-1/2}\, \Xi(\alpha^\ell,r)
		&\qquad& \text{for $\alpha^\ell\neq 0$},
		\\
		\abs{\mathcal{D}_{\ell,2}(\alpha^\ell,t)}
		&\lesssim t^{-1}\, \Xi(\alpha^\ell,r)
		&& \text{for $\alpha^\ell\neq 0$}.
	\end{alignedat}
\end{equation}

Next, we consider the large
time regime $\sqrt{2\overline{c}t}>|x_\alpha^\ell|$.
Note that in this regime, we have
$C\sqrt{r^\ell}\ge \abs{\alpha^\ell}$,
where $C\coloneqq\sqrt{\overline{c}/\min_j c^j}
=\sqrt{\max_j c^j/\min_j c^j}\ge 1$.
By \eqref{eq:nabla-plus-an}
and \eqref{eq:nabla-neg-an}---again keeping
in mind \eqref{eq:a-alphaj-cj}, \eqref{eq:sd-heat-c-1D},
and \eqref{eq:sol-1D-c}---we obtain
\begin{equation}\label{eq:D-ell-1-large}
	\abs{\mathcal{D}_{\ell,1}(\alpha^\ell,t)}
	\lesssim
	\begin{cases}
		\displaystyle
		\Dx^{-2}(r^\ell)^{-1/2}
		a_{\alpha^\ell}(r^\ell)
		\lesssim t^{-1/2}
		a_{\alpha^\ell}^{(c^\ell)}(t)
		& \text{for $\alpha^\ell\ge 1$},
		\\ \displaystyle
		\Dx^{-2}(r^\ell)^{-1/2}
		a_{\alpha^\ell}(r^\ell)
		\lesssim t^{-1/2}
		a_{\alpha^\ell}^{(c^\ell)}(t)
		& \text{for $\alpha^\ell\le -2$}.
	\end{cases}
\end{equation}
Similarly, by \eqref{eq:nabla-plus-squared-an}
and \eqref{eq:nabla-neg-squared-an},
\begin{equation}\label{eq:D-ell-2-large}
	|\mathcal{D}_{\ell,2}(\alpha^\ell,t)|
	\lesssim
	\begin{cases}
		\displaystyle
		\Dx^{-3} (r^\ell)^{-1}
		a_{\alpha^\ell}(r^\ell)
		\lesssim t^{-1}
		a_{\alpha^\ell}^{(c^\ell)}(t),
		& \text{for $\alpha^\ell\ge 1$},
		\\ \displaystyle
		\Dx^{-3} (r^\ell)^{-1}
		a_{\alpha^\ell+2}(r^\ell)
		\lesssim
		t^{-1}a_{\alpha^\ell+2}^{(c^\ell)}(t),
		& \text{for $\alpha^\ell\le -3$}.
	\end{cases}
\end{equation}
From \eqref{eq:fund_final_expression} we have
$$
a_{\alpha^\ell}^{(c^\ell)}(t)
=\Xi(\alpha^\ell,r^\ell)
\Psi(\alpha^\ell,r^\ell),
$$
where $\abs{\Psi(\alpha^\ell,r^\ell)}\lesssim 1$
(see \eqref{eq:Psinr-def}) and $\Xi(\alpha^\ell,r^\ell) \lesssim \Xi(\alpha^\ell,r)$, recalling that
$r^\ell$ is defined by \eqref{eq:fundamentalsolution}
and $r$ by \eqref{eq:def-r}.
Therefore, in the estimates \eqref{eq:D-ell-1-large}
and \eqref{eq:D-ell-2-large}, we
can estimate $a_{\alpha^\ell}^{(c^\ell)}(t)$
by $\Xi(\alpha^{\ell},r)$.

In the large time regime where
($\alpha^\ell\neq 0$)
\begin{equation}\label{eq:large-time-tmp}
	\sqrt{2\overline{c}t}>|x_\alpha^\ell|
	=\abs{\alpha^\ell} \Dx
	\quad \Longleftrightarrow \quad \frac{1}{\sqrt{r}} =
	\frac{\Dx}{\sqrt{2\overline{c}t}}
	<\frac{1}{\abs{\alpha^\ell}}\le 1,
\end{equation}
one can check that\footnote{
    We define $\Theta(x,t) := f\left(\frac{x}{4\overline{c}t}\right)$, where $f(x) := \frac{1}{1 + x^2 + |x|^3}$ and $\overline{c} > 0$. Let $y = \frac{x}{2\overline{c}t}$ and $\delta = \frac{\Delta}{2\overline{c}t}$, so that $\Theta(x,t) = f(y)$ and $\Theta(x+\Delta, t) = f(y + \delta)$. Now observe that the ratio
$$
\frac{f(y + \delta)}{f(y)} = \frac{1 + y^2 + |y|^3}{1 + (y + \delta)^2 + |y + \delta|^3}
$$
is continuous in $y$ and tends to $1$ as $|y| \to \infty$, since the leading terms in the numerator and denominator grow like $|y|^3$. Therefore, the ratio is bounded above on $\mathbb{R}$, and we can define $C(\delta)$ as this bound. If $\delta \le K$, we set $C(K) = \sup_{\delta \in (0,K]} C(\delta)$, which gives
$$
f(y + \delta) \le C(K) f(y),
$$
and hence $\Theta(x + \Delta, t) \le C(K) \Theta(x, t)$ for all $x$.
}, for any $\alpha^\ell\in \Z$,
$$
\Xi(\alpha^\ell+2,r)
\leq C(K)\Xi(\alpha^{\ell},r),
$$
where $C(K)$ is a constant dependent on $K$, where $K$ satisfies $\frac{2\Dx}{\sqrt{2\overline{c}t}} \le K$.
Note that we can set $C(K) = 1$ if $\alpha^\ell \ge 0$.
This condition is satisfied in the current large
time regime \eqref{eq:large-time-tmp}, as then we have
$\frac{2\Dx}{\sqrt{2\overline{c} t}}<2 \eqqcolon K$.
Therefore, in the estimate \eqref{eq:D-ell-2-large}
of $\mathcal{D}_{\ell,2}(\alpha^\ell,t)$,
we can replace $a_{\alpha^\ell+2}^{(c^\ell)}(t)$
by $\Xi(\alpha^{\ell},r)$.
Consolidating our findings, in the
large time regime \eqref{eq:large-time-tmp},
the following estimates apply:
\begin{equation}\label{eq:large-time-D-est}
	\begin{alignedat}{2}
		\abs{\mathcal{D}_{\ell,1}(\alpha^\ell,t)}
		&\lesssim t^{-1/2}\, \Xi(\alpha^\ell,r)
		&\qquad& \text{for } \alpha^\ell\neq -1,0,\\
		\abs{\mathcal{D}_{\ell,2}(\alpha^\ell,t)}
		&\lesssim t^{-1}\, \Xi(\alpha^\ell,r)
		&& \text{for } \alpha^\ell\neq -2,-1,0.
	\end{alignedat}
\end{equation}

We now turn our attention to the
exception cases
in \eqref{eq:small-time-D-est} and
\eqref{eq:large-time-D-est}. To begin, we focus on the
specific case of $\alpha^\ell = 0$
for the first-order difference term
$\mathcal{D}_{\ell,1}$. By \eqref{eq:a-alphaj-cj},
\eqref {eq:nabla-plus-a0},
our improved ``small time" estimate
\eqref{eq:pointwise_da}
($m=0$, $d=1$), and finally
that $t^{-1/2}\sim \Xi(0,r)
=1/\sqrt{2\overline{c}t}$ (see
\eqref{eq:Xinr-def}),
\begin{align*}
	\abs{\mathcal{D}_{\ell,1}(0,t)}
	& =\Dx^{-2}\abs{a_{0}(r^\ell)-a_{1}(r^\ell)}
	\\ &
	\overset{\eqref{eq:nabla-plus-a0}}{\lesssim}
	\Dx^{-2} (r^{\ell})^{-1}a_0(r^\ell)
	\sim t^{-1}a_0(r^\ell)
	\\ &
	=\Dx\, t^{-1}a_0^{(c^\ell)}(t)
	\overset{\mathclap{\eqref{eq:pointwise_da}}}{\lesssim}
	\Dx\, t^{-1} \bk{\frac1{\sqrt{2\overline{c}}}
	\wedge \frac{t^{1/2}}{\Dx}}\Xi(0,r)
	\\ &
	=\Xi(0,r)\times
	\begin{cases}
		\Dx\, t^{-1}\frac{t^{1/2}}{\Dx}=t^{-1/2}
		& \text{for
		$\frac{t^{1/2}}{\Dx}
		\le \frac{1}{\sqrt{2\overline{c}}}$},\\
		\Dx\,  t^{-1}\frac1{\sqrt{2\overline{c}}}
		\lesssim t^{-1/2}
		& \text{for
		$\frac{t^{1/2}}{\Dx}
		>\frac{1}{\sqrt{2\overline{c}}}$}
	\end{cases},
\end{align*}
so we have the desired bound $\abs{\mathcal{D}_{\ell,1}(0,t)}
\lesssim t^{-1/2} \Xi(0,r)$ for $\alpha^\ell=0$.
This bound holds for all $t>0$.

Next, we consider the case $\alpha^\ell= 0$
for the second order difference term
$\mathcal{D}_{\ell,2}$. By
\eqref{eq:time-decay-tmp} for $m=2$,
\begin{align*}
	\abs{\mathcal{D}_{\ell,2}(0,t)}
	& =\Dx^{-3}\abs{a_{0}(r^\ell)-2a_{1}(r^\ell)
	+a_{2}(r^\ell)}
	\\ &
	\leq \Dx^{-3}\, C_2\, (r^\ell)^{-3/2}
	\sim t^{-3/2}=t^{-1}t^{-1/2}
	\sim t^{-1} \Xi(0,r),
\end{align*}
where we have once again used that
$t^{-1/2}\sim \Xi(0,r)$. Therefore,
we obtain the sought-after estimate
$\abs{\mathcal{D}_{\ell,2}(0,t)}
\lesssim t^{-1} \Xi(0,r)$
for $\alpha^\ell=0$.
This bound, similar to the previous $\alpha^\ell=0$
bound for $\mathcal{D}_{\ell,1}$, holds
for all $t > 0$.

The remaining cases are the ones where
$\alpha^\ell=-1,-2\eqqcolon l$ in the large
time regime \eqref{eq:large-time-tmp}:
\begin{equation}\label{eq:large-time-tmp2}
	\sqrt{2\overline{c}t}>|x_\alpha^\ell|
	=\abs{l} \Dx
	\quad \Longleftrightarrow \quad
	\frac{\Dx}{\sqrt{2\overline{c}t}}<\frac{1}{\abs{l}}
	\le 1.
\end{equation}

We begin with the second order difference term
$\mathcal{D}_{\ell,2}(l,t)$.
Similar to the previous case, we deduce the
following bound for $l=-2,-1$:
\begin{align*}
	\abs{\mathcal{D}_{\ell,2}(l,t)}
	& =\Dx^{-3}\abs{a_{l}(r^\ell)-2a_{l+1}(r^\ell)
	+a_{l+2}(r^\ell)}
	\\ &
	\overset{\mathclap{\eqref{eq:time-decay-tmp}}}{\leq}
	\Dx^{-3}\, C_2\, (r^\ell)^{-3/2}
	\sim t^{-3/2}=t^{-1}t^{-1/2}
	\sim t^{-1} \Xi(0,r).
\end{align*}
One can check that for $l<0$,
\begin{equation}\label{eq:Xi-null-l}
	\Xi(0,r) \leq C(K)\Xi(l,r),
\end{equation}
where $C(K)$ is a constant dependent on $K$, given
by $\frac{\abs{l}\Dx}{\sqrt{2\overline{c}t}} \le K$.
This condition is satisfied in the large
time regime \eqref{eq:large-time-tmp2},
because then we have
$$
\frac{\abs{l}\Dx}{\sqrt{2\overline{c} t}}
<\frac{1}{\sqrt{2}}\eqqcolon K.
$$
Therefore, in the large time
regime \eqref{eq:large-time-tmp2},
\begin{align*}
	\abs{\mathcal{D}_{\ell,2}(l,t)}
	\overset{\eqref{eq:Xi-null-l}}{\lesssim}
	t^{-1} \Xi(l,r),\qquad l=-2,-1.
\end{align*}

Finally, we consider the first order
difference term $\mathcal{D}_{\ell,1}(l,t)$
for $l=-1$. By \eqref{eq:time-decay-tmp} for $m=1$,
\begin{align*}
	\abs{\mathcal{D}_{\ell,l}(l,t)}
	& =\Dx^{-2}\abs{a_{0}(r^\ell)-a_{1}(r^\ell)}
	\\ &
	\lesssim \Dx^{-2}
	\, C_1\, (r^\ell)^{-3/2}
	\sim \Dx \, t^{-3/2}
	= \Dx\, t^{-1}t^{-1/2}
	\\ &
	= \Dx\, t^{-1}\Xi(0,r)
	= \Dx\, t^{-1/2} t^{-1/2}\Xi(0,r).
\end{align*}
In the large time regime \eqref{eq:large-time-tmp2},
$\Dx\, t^{-1/2}\lesssim 1$ and so
\begin{align*}
	\abs{\mathcal{D}_{\ell,1}(l,t)}
	\lesssim
	t^{-1/2}\Xi(0,r)
	\overset{\eqref{eq:Xi-null-l}}{\lesssim}
	t^{-1/2}\Xi(l,r), \qquad l=-1.
\end{align*}

In summary, from \eqref{eq:small-time-D-est} and
\eqref{eq:large-time-D-est}, along with the previous
analysis of the exception cases, we have established
that for $t>0$ and $\alpha^\ell\in\Z$,
$$
\abs{\mathcal{D}_{\ell,1}(\alpha^\ell,t)}
\lesssim t^{-1/2}\Xi(\alpha^\ell,r)
$$
and
$$
\abs{\mathcal{D}_{\ell,2}(\alpha^\ell,t)}
\lesssim t^{-1}\Xi(\alpha^\ell,r).
$$
By inserting these estimates into \eqref{eq:m-deriv}
and applying the splitting $1=\one{F}+\one{F^{\tt c}}$,
see \eqref{eq:F-def}, \eqref{eq:Fc-def},
and \eqref{eq:a_F_bound1}, we
derive the following result for $m=1,2$:
\begin{align*}
	\abs{(\nabla_+^\ell)^m a_\alpha(t)}
	\lesssim t^{-m/2}\left(\,
	\prod_{j = 1}^d \Xi(\alpha^j,r)
	\one{F}(\alpha^j,r)
	+\frac1\Dx\one{F^{\tt c}}(\alpha^j,r)\right).
\end{align*}
We can now repeat the argument that
was used to estimate \eqref{eq:a_F_bound1},
eventually leading to
\begin{align*}
	&\abs{(\nabla_+^\ell)^m a_\alpha(t)}\lesssim t^{-m/2}
	\left(\bk{\frac1{\sqrt{2\overline{c}}}
	\wedge \frac{ t^{1/2}}{\Dx}}^{Z(\alpha)}
	t^{-d/2}\prod_{j=1}^d
	\bk{1+\frac{\abs{x_\alpha^j}^2}{2\overline{c}t}
	+\frac{\abs{x_\alpha^j}^3}{\abs{2\overline{c}t}^{3/2}}}^{-1}
	\right),
\end{align*}
which is \eqref{eq:pointwise_da} for $m=1,2$.

Finally, given \eqref{eq:pointwise_da} for $m=1,2$, we
can proceed as in the $m = 0$ case---see
\eqref{eq:comp-int-test}---to derive the $L^p$ estimate
in \ref{eq:a_lp}. The only distinction
lies in the inclusion of the additional factor $t^{-m/2}$.
\end{proof}

%

\begin{remark}
The same estimate \eqref{eq:pointwise_da} holds for
backward differences, as well as for alternating forward and
backward differences. Therefore, for
$\alpha\in \Z^d$ and $t>0$,
\begin{equation}\label{eq:a_F_bound3}
	\begin{aligned}
		\abs{\nabla_+^\ell\nabla_-^\ell a_\alpha(t)}
		\lesssim~&\biggl(\frac1{\sqrt{2\overline{c}}}
		\wedge \frac{ t^{{1/2}}}{\Dx}\biggr)^{Z(\alpha)}
		\\ &\times
		t^{-\bk{d + 2}/2}\prod_{j = 1}^d
		\bk{1 + \frac{\abs{x_\alpha^j}^2}{2 \overline{c}t}
		+\frac{\abs{x_\alpha^j}^3}{\abs{2\overline{c}t}^{3/2}} }^{-1}.
	\end{aligned}
\end{equation}
\end{remark}

Using Proposition \ref{lem:green_estimates}, we
can also solve the (anisotropic)
constant-coefficient semi-discrete
heat equation with a source.

\begin{lemma}[Solution formula, constant coefficients]
\label{lem:duhamel_representation}
Fix $p,q$ satisfying $p > d$ and $2/q+d/p<1$.
Let $c^1,\ldots,c^d>0$, $\bm{f}
\in L^q\bigl([0,T];L^p_\Dx(\R^d)\bigr)$,
and $\bm{\psi} \in L^\infty_\Dx(\R^d)$.
Then the unique solution $\bm{u} \in C^1([0,T];L^\infty_\Dx(\R^d))$
of the inhomogeneous
semi-discrete heat equation
$$
\begin{cases}\displaystyle
	\frac{\d}{\d t}   u_\alpha
	-\sum_{j = 1}^d c^j \nabla_+^j \nabla_-^j u_\alpha
	=f_\alpha, & t \in [0,T], \alpha \in \Z^d,
	\\
	u_\alpha(0) = \psi_\alpha, & \alpha\in\Z^d
\end{cases}
$$
is given by the Duhamel formula
$$
 u_\alpha(t)
\coloneqq \bk{\psi \dconv a(t)}_\alpha
+ \int_0^t \bk{a(t - s) \dconv f(s)}_\alpha \,\d s, \qquad t > 0,
\,\, \alpha \in \Z^d.
$$
This function satisfies
\begin{equation*}
	\begin{split}
		\|\bm{ u}(t)\|_{L^\infty(\R^d)}
		&\leq \|\bm{\psi}\|_{L^\infty(\R^d)}
		+ C\|\bm{f}\|_{L^q([0,T];L^p (\R^d))},  \\
		\|\nabla_+ \bm{ u}(t)\|_{L^\infty (\R^d)}
		&\leq \|\nabla_+ \bm{\psi}\|_{L^\infty(\R^d)}
		+ C\|\bm{f}\|_{L^q([0,T];L^p (\R^d))},
	\end{split}
\end{equation*}
for some $C>0$ that is independent of $\Dx$.
\end{lemma}

The proof is identical to that presented in
\cite[Lemma 4.3]{Fjordholm:2023aa}, and
is therefore omitted here.
In particular, the bounds on ${\bm u}$ and its first
order difference follow from the Duhamel representation
and Young's inequality in time
and space (Lemma \ref{lem:young}).

\subsection{Gaussian estimates away
from $t = 0$}\label{sec:gaussian_estm}

We conclude this section with a (standard)
proof that pointwise Gaussian estimates on
the Green's function $a_\alpha$ hold for
$t \gtrsim \alpha \Dx^2$. These 
estimates are included for completeness, and 
we will not be invoking them subsequently. We start with a technical lemma.

\begin{lemma}\label{thm:oscillatory_integrand}
Let $z = \theta + i b$ for $\theta,b\in
[-\hf,\hf]$, let $\alpha \in \Z$ and $c > 0$,
and set
\begin{align}\label{eq:F_m_defin}
	\mathfrak{F}_m(z,\alpha, c, t)
	\coloneqq \exp\bk{ -\frac{4ct}{\Dx^2}
	\sin^2(\pi z)} e^{2\pi i \alpha z} \bk{e^{2 \pi i z} - 1}^m.
\end{align}
Then
\begin{align*}
	\abs{\mathfrak{F}_m(z,\alpha,c,t)}
	\le 2^m \exp\bk{ \frac{4ct}{\Dx^2}
	\bk{-\theta^2 + 63 b^2} - \alpha b}.
\end{align*}
\end{lemma}

\begin{proof}
We write $\mathfrak{F}_m$ as a
product of five factors $\mathfrak{F}^{(k)}$ thus:
\begin{align*}
	&\mathfrak{F}_m(z,\alpha,c,t)
	\eqqcolon \prod_{k = 1}^5 \mathfrak{F}^{(k)},
\end{align*}
where
\begin{align*}
	&\mathfrak{F}^{(1)}\coloneqq \exp\bk{\frac{-4ct}{\Dx^2}
	\sin^2(\pi \theta)},\qquad
	\mathfrak{F}^{(2)}\coloneqq \exp\bk{\frac{4ct}{\Dx^2}
	\bk{\sin^2(\pi \theta)- \sin^2(\pi z)}},
	\\ &
	\mathfrak{F}^{(3)}\coloneqq e^{2\pi i \alpha
	\bk{z - \theta}},
	\qquad
	\mathfrak{F}^{(4)}\coloneqq e^{2\pi i \alpha \theta},
	\qquad
	\mathfrak{F}^{(5)}\coloneqq  \bk{e^{2 \pi i z} - 1}^m.
\end{align*}
We have directly that
\begin{align}\label{eq:F34}
	\abs{\mathfrak{F}^{(3)}} = e^{-2 \pi\alpha b},
	\qquad
	\abs{\mathfrak{F}^{(4)}} = 1,\qquad
	\abs{\mathfrak{F}^{(5)}} \le \bk{e^{- 2\pi b} + 1}^m
	\le 2^m.
\end{align}
Since $\sin^2(\pi \theta) \ge \theta^2$,
we also have
\begin{align}\label{eq:F1}
	\abs{\mathfrak{F}^{(1)}}
	\le \exp\bk{\frac{-4ct}{\Dx^2}\theta^2}.
\end{align}
Finally, denoting by $\mathfrak{R}z$
the real part of $z$,
\begin{align*}
	&\bigl|\mathfrak{R}(\sin^2(\pi \theta) - \sin^2(\pi z) )\bigr|
	\\ &
	= \biggl|\sin^2(\pi \theta)
	+\frac14 \mathfrak{R} \bk{e^{2\pi i \theta} e^{-2\pi b}
	+ e^{-2\pi i \theta} e^{2\pi b} - 2}\biggr|
	\\ &
	= \biggl| \sin^2(\pi \theta)
	+\frac12 \bk{ \cos(2 \pi \theta)\cosh(2\pi b) - 1} \biggr|
	\\ &
	\le \sin^2(\pi \theta) \abs{1 - \cosh(2 \pi b)}
	+\frac12 \abs{1 - \cosh(2 \pi b)}
	\le \frac32  \abs{1 - \cosh(2 \pi b)},
\end{align*}
where we used the elementary
trigonometric identity
$$
\cos(2 \pi \theta) = 1- 2 \sin^2(\pi \theta).
$$
For $b \in [-\hf,\hf]$, we have
$\abs{1 - \cosh(2 \pi b)}\le  42 b^2$.
This gives
\begin{align*}
	\abs{\mathfrak{F}^{(2)}}
	\le e^{\frac{252ct}{\Dx^2} b^2}.
\end{align*}
Together with \eqref{eq:F34} and \eqref{eq:F1},
we get the desired bound.
\end{proof}
%
In the large time regime ($t$ bounded away 
from $0$), we can in fact establish a Gaussian --- as opposed 
to algebraic --- tail (cf. Proposition \ref{lem:green_estimates} (iii)):
\begin{proposition}\label{thm:a_pointwise_exponential}
Let $a$ be given by \eqref{eq:greensfunction}.
Then, for any
$b^j \in [-\hf,\hf]$, we have the estimate
\begin{align}\label{eq:a_expbound_b}
	\abs{a_\alpha(t)}
	&\le \pi^{d/2}\prod_{j = 1}^d \frac1{\sqrt{4 c^j t}}
	\exp\biggl({\frac{C_0c^j t}{2\Dx^2}
	|b^j|^2 - \alpha^j b^j }\biggr)
\end{align}
where $C_0\coloneqq 252$.
In particular, for $t \ge \max_{j \le d}\abs{\alpha^j}
\Dx^2/\bk{ \min_{j \le d} 2C_0c^j}$, we have
the pointwise estimate
\begin{align}\label{eq:a_expbound_alpha}
	\abs{a_\alpha(t)}
	&\le \pi^{d/2} \prod_{j=1}^d\frac{1}{\sqrt{4 c^j t}}
	\exp\biggl(- \frac{(\alpha^j\Dx)^2}{2C_0 c^j t}\biggr).
\end{align}
\end{proposition}

\begin{proof}
Following \cite{Hoff:1985zm}, from the semi-discrete heat
equation \eqref{eq:discheateq}, and
using summation-by-parts, we obtain
\begin{align*}
	\frac{\d}{\d t}\widehat{a}(t,\theta)
	&= \sum_{\alpha \in \Z^d}
	\nabla_+ \cdot \bigl(c\nabla_- a_\alpha(t)\bigr)
	 e^{2 \pi i \alpha \cdot \theta}
	= \sum_{\alpha \in \Z^d}
	a_\alpha(t)\nabla_- \cdot \bigl(c\nabla_+
	 e^{2 \pi i \alpha \cdot \theta}\bigr) \\
	& = \sum_{\alpha \in \Z^d}a_\alpha(t)
	e^{2 \pi i \alpha \cdot \theta}
	\sum_{j=1}^d c^j\frac{e^{2\pi i\theta^j} - 2
	+ e^{-2\pi i\theta^j}}{\Dx^2}\\
	& = \Biggl(\sum_{j = 1}^d 2c^j
	\frac{\cos(2\pi\theta^j)-1}{\Dx^2}\Biggr)
	\widehat{a}(t,\theta)
	= \Biggl(-\sum_{j = 1}^d \frac{4c^j}{\Dx^2}
	\sin^2(\pi \theta^j)\Biggr)
	\widehat{a}(t,\theta),
\end{align*}
where we used $\cos(2 \pi \theta^j) - 1
= 2 \sin^2(\pi \theta^j)$ in the final line.
By the fact that $\widehat{a}(0,\theta)=\Dx^{-d}$, we obtain
\begin{equation}\label{eq:w_aux2}
	\begin{split}
		\widehat{a}(t,\theta)
		&= \frac{1}{\Dx^d}\exp\Biggl(-\sum_{j = 1}^d
		\frac{4 c^jt}{\Dx^2}
		\sin^2(\pi \theta^j)\Biggr).
	\end{split}
\end{equation}
Inverting the Fourier transform,
\eqref{eq:w_aux2} yields the following oscillatory integral:
\begin{align*}
	a_\alpha(t)
	&=\frac{1}{\Dx^d} \int_{[-\hf, \hf)^d}
	\exp\Biggl(-\sum_{j = 1}^d \frac{4 c^jt}{\Dx^2}
	\sin^2(\pi \theta^j)\Biggr)
	e^{2 \pi i \alpha\cdot \theta} \,\d\theta\\
	& = \prod_{j = 1}^d \frac1\Dx \int_{-\hf}^{\hf}
	\underbrace{\exp\biggl(-\frac{4 c^jt}{\Dx^2}
	\sin^2(\pi \theta^j)\biggr)
	e^{2 \pi i \alpha^j \theta^j}}_{
	=\, \mathfrak{F}_0(\theta^j,\alpha^j,c^j,t)} \,\d\theta^j.
\end{align*}
Here, $\mathfrak{F}_0$ is defined in \eqref{eq:F_m_defin}.

We now derive precise bounds on $\abs{a_\alpha}$
in the large time regime $t \gtrsim \abs{\alpha} \Dx^2$.
Since $\alpha^j \in \Z$, $\mathfrak{F}_0$ is a $1$-periodic
function of its first argument $\theta^j$.
Moreover, allowing $\theta^j$ to range
over $\mathbb{C}$, $\mathfrak{F}_0$ is
a composition and product of
holomorphic functions and hence
a holomorphic function of $\theta^j$.
By Cauchy's integral theorem, for any $b^j \in \R$,
we can integrate on a contour around the
rectangle with real parts between $-\hf$
and $\hf $, and imaginary parts
between $0$ and $b^j$ to get:
\begin{align*}
	\int_{-\hf + ib^j}^{\hf + ib^j}
	\mathfrak{F}_0(z,\alpha^j,c^j,t)\,\d z
	&=\int_{-\hf}^{\hf}\mathfrak{F}_0(z,\alpha^j,c^j,t)\,\d z
	\\ &\quad\,\,
	+\underbrace{\int_{\hf}^{\hf + ib^j}
	+\int_{-\hf + ib^j}^{-\hf}
	\mathfrak{F}_0(z,\alpha^j,c^j,t)\,\d z}_{=\, 0},
\end{align*}
where the final two integrals sum
to nought by $1$-periodicity.

The advantage of this man{\oe}uvre is that
shifting the integration onto the line
$z \in [-\hf  + ib^j,\hf  + ib^j]$ introduces an exponential
decay in the factor  $e^{2\pi i \alpha^j z}$ for any
$b^j > 0$. Indeed, we have
the following bound for $b^j  \in [0, \hf ]$ by
Lemma \ref{thm:oscillatory_integrand}:
\begin{align*}
	\abs{\mathfrak{F}_0(\theta^j + i b^j, \alpha^j c^j, t) }
	\le \exp\bk{ \frac{4c^j t}{\Dx^2} \bk{- \bk{\theta^j}^2
	+63 \bk{b^j}^2} - \alpha^j b^j}
\end{align*}
Therefore,
\begin{align*}
	\abs{a_\alpha(t)}
	&\le \prod_{j = 1}^d \frac1\Dx
	\exp\biggl(\frac{252c^j t}{\Dx^2}
	(b^j)^2 - \alpha^j b^j \biggr)
	\int_{-\infty}^{\infty}
	\exp\bk{ -\frac{4c^j t}{\Dx^2}
	\bk{\theta^j}^2 } \,\d\theta^j
	\\ &
	=\prod_{j=1}^d \frac1\Dx
	\exp\biggl(\frac{252c^j t}{\Dx^2}(b^j)^2
	-\alpha^j b^j \biggr)
	\sqrt{\frac{\pi\Dx}{4c^jt}},
\end{align*}
which is \eqref{eq:a_expbound_b}. The above expression
is minimized when $b^j = \frac{\alpha^j \Dx^2}{C_0c^j t}$,
and if $t \ge \alpha^j \Dx^2/\bk{\min_{j \le d} 2C_0c^j}$,
then indeed $\abs{b} \le \hf$. This leads
to \eqref{eq:a_expbound_alpha}.
\end{proof}


\subsection{Comments on the heat kernel estimates
of Davies and Pang}\label{sec:daviespang}

Let us briefly compare our heat kernel estimates
\eqref{eq:pointwise_da} (for $m=0$)
with those presented in \cite{Pang:1993aa}.
Consider the one-dimensional semi-discrete heat kernel
$a_n(t)$ defined by \eqref{eq:sol-1D}
for all $n \in \Z$ and $t > 0$, which represents
the solution to the semi-discrete heat equation
\eqref{eq:sd-heat-1D}, assuming a unit diffusion
coefficient and unit grid spacing.
In \cite[Theorem 3.5]{Pang:1993aa},
the author demonstrated that there exists a
constant $C \geq 1$ such that
\begin{align}\label{eq:pangs_estimate}
	a_n(t) \leq
	\begin{cases}
		C n^{-1/2}
		\exp\Bigl(\abs{n}
		F\left(\frac{\abs{n}}{2t}\right)\Bigr)
		&\text{if }  0 < t \leq \abs{n},
		\\
		C t^{-1/2}
		\exp\Bigl(\abs{n}
		F\left(\frac{\abs{n}}{2t}\right)\Bigr)
		&\text{if }   0 < \abs{n} \leq t,
	\end{cases}
\end{align}
where $F$ is given by
\begin{align}\label{eq:dp_F}
	F(\gamma)
	=-\log\Bigl(\gamma+\sqrt{\gamma^2+1}\Bigr)
	+\frac{\sqrt{\gamma^2+1}-1}{\gamma},
	\quad \gamma>0.
\end{align}
We have
$F(\gamma) \leq -\frac{\gamma}{2} + \frac{\gamma^3}{20}$,
which for small $\gamma$ leads to the usual
Gaussian upper bound
\begin{equation}\label{eq:pang-exp-est}
	a_n(t) \lesssim
	t^{-1/2}e^{-\abs{\alpha}^2/(4t)}
	\qquad \text{for $t \gg \abs{n}$}.
\end{equation}
Additionally, $F(\gamma) \leq
-\log\Bigl(\frac{2\gamma}{e}\Bigr)$,
implying the following
non-Gaussian (log-corrected) behaviour:
\begin{equation}\label{eq:log-est}
    a_n(t) \lesssim t^{-1/2}
    \exp\Bigl(-\abs{n}
    \log\Bigl(\tfrac{\abs{n}}{e t}\Bigr)\Bigr)
    \lesssim t^{-1/2}
    \left(\frac{et}{\abs{n}}\right)^{\abs{n}},
\end{equation}
for $t>0$ and $n\in \Z \setminus \{0\}$.
This indicates a significant difference in the
short-time asymptotics between the continuous
and semi-discrete heat kernels.

There is a straightforward
probabilistic interpretation of
the leading order term in \eqref{eq:log-est}.
Consider a continuous-time random walk on $\Z$
that starts at the origin, with a $\hf$ probability
of jumping to either neighbour. The corresponding
jump probabilities of this Markov process are
given by $a_n(t)$.  Because at each state the walk
can jump in two equally likely directions,
each with an ``exponential clock" of rate $1$,
these rates add up to yield a total jump rate of $2$ at every state.
It follows that in a small time increment $\Delta t$ the probability
of a jump is $1 - e^{-2\Delta t} \approx 2\Delta t$, and hence
partitioning $[0,t]$ into $|n|$ intervals,
each of length $\Delta t = t/|n|$,  yields an approximate
jump probability of $\bigl(2\Delta t\bigr)^{|n|}
=\bigl(2t/|n|\bigr)^{|n|}$. This aligns with the last term
in \eqref{eq:log-est}, here with $2$ instead of $e$.

Next, let us examine the solution $a_n^{(1)}(t)$
of the semi-discrete heat equation
\eqref{eq:sd-heat-c-1D} on $\Dx \Z$ with $c=1$
and $\Dx>0$. Recall from \eqref{eq:sol-1D-c} that
$$
a_n^{(1)}(t)=\frac{1}{\Dx}a_n(2t/\Dx^2),
$$
where $a_n(t)$ is the solution of
\eqref{eq:sd-heat-1D} discussed above.
Directly rescaling
\eqref{eq:pang-exp-est} gives a Gaussian estimate
$a_n^{(1)}(t) \lesssim t^{-1/2}
e^{-\abs{x_n}^2/(8t)}$
for large times $t \gg \abs{x_n}\Dx$, $x_n=n\Dx$,
and $n\in \Z\setminus \{0\}$.
This can be compared to our
\eqref{eq:a_expbound_alpha}, which
holds for $t\ge|x_n|\Dx$ and $n\in \Z$.
For small times, specifically when $0 < t < |x_n| \Delta x$
(see the first part of \eqref{eq:pangs_estimate}),
rescaling \eqref{eq:log-est} yields the estimate:
$a_n^{(1)}(t) \lesssim t^{-1/2}
\left(\frac{2et}{|x_n| \Delta x}\right)^{|n|}$,
valid for $t > 0$ and $n \in \Z \setminus \{0\}$.

Considering very small $t$, say $t < \Delta x / (2e)$,
we obtain $a_n^{(1)}(t) \lesssim t^{-1/2} |x_n|^{-|n|}$.
This expression demonstrates strong decay as
$|n| \to \infty$ for fixed $t$, along with a $t^{-1/2}$
behavior as $t \to 0$ for finite, nonzero $n$.
This should be compared with our
\eqref{eq:pointwise_da} (for $m=0$)---see
also \eqref{eq:pointwise_da-intro}
and the accompanying discussion---which
holds for all $n\in \Z$ but has a weaker
spatial decay. Indeed,
\eqref{eq:pangs_estimate} does not
address the case $n = 0$, which particularly
implies the absence of a bound at the origin for
small times. On the other hand, our bound in
\eqref{eq:pointwise_da} encompasses this scenario, providing
the correct $\frac{1}{\Delta x}$ bound at the origin.

Therefore, we will rely on \eqref{eq:pointwise_da} in
constructing the parametrix in the following sections.
Although \eqref{eq:pointwise_da} is sub-optimal regarding
its $n$-decay within the regime covered by \eqref{eq:pangs_estimate}, it
remains preferable due to its coverage of the $n = 0$ case.
Additionally, computing the Neumann series in
the parametrix approach, using the heat kernel
bound \eqref{eq:pangs_estimate} and the
function $F$ defined in \eqref{eq:dp_F},
currently presents challenges.

\section{Variable coefficients}\label{sec:variable_coeff}
Let $\bm{f}\from[0,\infty)\to L^\infty_\Dx(\R^d)$
and $\bm{\psi}, \bm{c}^1, \ldots,
\bm{c}^d \in L_\Dx^\infty(\R^d)$ be
given grid functions and define
${\bm c}= \text{diag}(\bm{c}^1, \ldots, \bm{c}^d)$.
Assume that $\inf_{\alpha,j}c_\alpha^j>0$, and that
\begin{align}\label{eq:smoothrho}
	\frac{\bigl|c_\alpha^j - c_\beta^j\bigr|}
	{\bigl|x_\alpha - x_\beta\bigr|}\le C,
	\quad j=1,\ldots d,
\end{align}
for a constant $C$ independent of $\alpha, \beta, j$.
We keep the convention that the superscript
$j = 1, \ldots, d$ of $c_\alpha^j$
are indices and not exponents. We consider here the anisotropic
\emph{variable coefficient} semi-discrete heat equation
\begin{equation}\label{eq:heat_variable_coeff}
	\begin{cases}\displaystyle
		\frac{\d}{\d t}  u_\alpha
		-\sum_{j = 1}^dc_\alpha^j
		\nabla_+^j \nabla_-^ju_\alpha
		=f_\alpha, & \alpha\in\Z^d,\ t>0,
		\\
		u_\alpha(0) = \psi_\alpha, & \alpha\in\Z^d.
	\end{cases}
\end{equation}

We will use the estimates provided
in Section \ref{sec:greens_functions}
in the variable coefficient case by
fixing the coefficient argument $\alpha\in \Z^d$ of the
coefficient $c_\alpha^j$ at a fixed index $\beta$.
We then get a constant coefficient
heat operator, and a resultant first
approximation $a_{\alpha,\beta}$
to the fundamental solution of the
full variable coefficient operator.

For functions $a\from \Z^d \times \Z^d \to \R$
mapping $(\alpha, \beta)$ to $a_{\alpha,\beta}$, we
will always consider $\beta$ as a fixed parameter,
allowing the difference operators
to act only on $\alpha$.
For ${\bm c}\in L^\infty_\Dx(\R^d)$ as
in \eqref{eq:smoothrho}, we define
$a=a_{\alpha, \beta}(t)$ as follows:
For fixed $\beta\in\Z^d$, the map
$(\alpha, t)\mapsto a_{\alpha,\beta}(t)$
is the fundamental solution (heat kernel) of
the following equation with
constant coefficients $c_\beta^{(1)},
\ldots,c_\beta^{(d)}$:
\begin{align}\label{eq:frozencoeff}
\frac{\d}{\d t}   u_\alpha
-\sum_{j=1}^d c_\beta^j
\nabla_+^j \nabla_-^ju_\alpha=0,
\quad
(t,\alpha)\in (0,T]\times \Z^d.
\end{align}
The theory presented
in Section \ref{sec:greens_functions}
establishes the existence, uniqueness, and
diverse properties of the heat kernel
$(\alpha, t)\mapsto a_{\alpha,\beta}(t)$
for each fixed $\beta$.

We will begin by quantifying the error
made by freezing $\alpha$ in
the argument of ${c_\alpha^j}$ at
$\alpha=\beta$. The next lemma offers an estimate
of the semi-discrete counterpart of the
continuous object \eqref{eq:intro-Phi}, which
characterizes the error between the parametrix and the
actual heat kernel (see \eqref{eq:intro-error}
and \eqref{eq:intro-parametrix}).

\begin{lemma}[Auxiliary estimates]
\label{thm:K_estimates_lemma}
Fix $\beta \in \Z^d$.
Define the auxiliary kernel
\begin{align}\label{eq:K_defin}
	K_{\alpha, \beta}(t)
	&\coloneqq \sum_{j = 1}^d \big(c_\alpha^j-c_\beta^j\big)
	\nabla_+^j \nabla_-^j{\, a_{\alpha -\beta, \beta}(t)},
\end{align}
where the difference operators $\nabla_\pm$
act only on the $\alpha$ subscript. If
$Z(\alpha)$ denotes the number of entries
of $\alpha \in \Z^d$ that are zero, we have
\begin{equation}\label{eq:K_estimate}
	\begin{aligned}
		|K_{\alpha, \beta}(t)|
		&\lesssim \bk{\frac1{\sqrt{2\overline{c}}}
		\wedge \frac{ t^{1/2}}{\Dx}}^{Z(\alpha - \beta)}
		t^{-\bk{d + 1}/2} \prod_{j = 1}^d
		\bigg(1 + \frac{\big|x_\alpha^j
		-x_\beta^j\big|^2}{2\overline{c}t}\bigg)^{-1},
	\end{aligned}
\end{equation}
where  $\overline{c} = \norm{\bm c}_{L^\infty(\R^d)}$,
and the implicit constant is independent of
$\alpha$, $\beta$, $\Dx$, and $t \in [0,\infty)$.
\end{lemma}

\begin{proof}
By Proposition \ref{lem:green_estimates}
\ref{lem:green_estimates_derivative-bound},
and in particular, using \eqref{eq:a_F_bound1}
and \eqref{eq:a_F_bound3},
with $\overline{c} \coloneqq
\sup_{j,\alpha} c^j_\alpha$,
\begin{align*}
	&\abs{K_{\alpha, \beta}(t)}\\*
	&\qquad \le  \frac{|x_\alpha - x_\beta|}{\sqrt{t}} \sqrt{t}
	\max_j \frac{\bigr|c_\alpha^j
	-c_\beta^j\bigl|}{{\big|x_\alpha - x_\beta\big|}}
	\abs{ \sum_{j = 1}^d \nabla_+^j \nabla_-^j
	a_{\alpha - \beta, \beta}}\\
	&\qquad\lesssim  \frac{|x_\alpha - x_\beta|}{\sqrt{t}}
	\Bigg[
	t^{- \bk{d + 1}/2} \prod_{j = 1}^d
	{\Biggl(1 + \frac{{\big|x_\alpha^j
	-x_\beta^j\big|}^2}{2 \overline{c}t}
	+\frac{\big|x_\alpha^j - x_\beta^j\big|^3}
	{\abs{2 \overline{c}t}^{3/2}} \Biggr)}^{-1}
	\one{\{2 \overline{c} t \ge \Dx^2\}}\\*
	&\qquad\relspace + \frac{t^{- {1/2}}}{\Dx^{Z(\alpha)}}
	\prod_{j:\alpha^j \neq \beta^j} t^{-{1/2}}
	{\Biggl(1 + \frac{{\big|x_\alpha^j - x_\beta^j\big|}^2}{2 \overline{c}t}
	+\frac{{\big|x_\alpha^j - x_\beta^j\big|}^3}
	{\abs{2 \overline{c}t}^{3/2}} \Biggr)}^{-1}
	\one{\{2 \overline{c} t <  \Dx^2\}}
	\Bigg].
\end{align*}
We now use the $3/2$ order term to absorb
the factor ${|x_\alpha - x_\beta|}/{\sqrt{t}}$:
\begin{align*}
	&\frac{|x_\alpha - x_\beta|}{\sqrt{t}}
	\prod_{j : \alpha^j \neq \beta^j}
	\Biggl(1 + \frac{{\big|x_\alpha^j - x_\beta^j\big|}^2}{2 \overline{c}t}
	+ \frac{{\big|x_\alpha^j
	-x_\beta^j\big|}^3}{\abs{2 \overline{c}t}^{3/2}} \Biggr)^{-1}
	\\ & \quad
	\lesssim
	\max_{k \le d}\frac{|x_{\alpha}^k
	-x_{\beta}^k|}{\sqrt{t}}
	\prod_{j : \alpha^j \neq \beta^j}
	\Biggl(1 + \frac{{\big|x_\alpha^j
	-x_\beta^j\big|}^2}{2 \overline{c}t}
	+\frac{{\big|x_\alpha^j - x_\beta^j\big|}^3}
	{\abs{2 \overline{c}t}^{3/2}} \Biggr)^{-1}
	\\ & \quad
	\lesssim
	\prod_{j : \alpha^j \neq \beta^j}
	\Biggl( 1 + \frac{{\big|x_\alpha^j
	-x_\beta^j\big|}^2}{2 \overline{c}t} \Biggr)^{-1}.
\end{align*}
Directly from the definition of $K$, we
see that $K_{\alpha, \beta} = 0$ when $\alpha = \beta$.
Therefore, for $2 \overline{c} t \ge \Dx^2$,
our estimate on $K_{\alpha, \beta}(t)$
reduces to:
$$
\abs{K_{\alpha, \beta}(t)}
\one{\{2 \overline{c} t \ge \Dx^2\}}
\lesssim
t^{-\bk{d + 1}/2} \prod_{j = 1}^d
\Biggl(1 + \frac{{\big|x_\alpha^j
-x_\beta^j\big|}^2}{2 \overline{c} t}\Biggr)^{-1}
\one{\{ 2 \overline{c} t \ge \Dx^2,
\alpha \neq \beta\}}.
$$
On the other hand, for  $2 \overline{c} t<\Dx^2$, we get
\begin{align*}
	\abs{K_{\alpha,\beta}(t)}
	\one{\{2 \overline{c} t <  \Dx^2\}}
	\lesssim \frac{t^{-{1/2}}}{\Dx^{Z(\alpha)}}
	\prod_{j : \alpha^j \neq \beta^j} t^{-{1/2}}
	\Biggl( 1 + \frac{\big|x_\alpha^j
	-x_\beta^j\big|^2}{2 \overline{c} t}\Biggr)^{-1}
	\one{\{2 \overline{c} t < \Dx^2,\alpha \neq \beta\}}.
\end{align*}
This establishes \eqref{eq:K_estimate}.
\end{proof}

Set $K^{(1)} \coloneqq K$, and
for $m = 2, 3, \ldots$, define
\begin{equation}\label{eq:Km_defin}
	K^{(m)}_{\alpha, \beta}(t)
	\coloneqq \int_0^t
	\Bigl( K(t - s)\dconv
	 K^{(m - 1)}(s)\Bigr)_{\alpha, \beta}\,\d s.
\end{equation}
A key component in the parametrix method, as illustrated
in the continuous analog
\eqref{eq:intro-parametrix}
(see, for example, \cite[Chapter 9,
Equation (4.8)]{Fri1964}), is the following sum of
iterated ``convolutions" of $K$:
\begin{equation}\label{eq:Phi_defin}
	\Phi_{\alpha, \beta}(t)
	\coloneqq \sum_{m= 1}^\infty K^{(m)}_{\alpha, \beta}(t).
\end{equation}

One of the primary technical contributions
of this work is the following result concerning
the convergence and pointwise bound of $\Phi$.
We dedicate the entire Section \ref{sec:aux_calc}
to its proof.

\begin{proposition}\label{thm:Phi_bound}
The series \eqref{eq:Phi_defin}
converges absolutely in $L^1([0,T])$ for
every $\alpha, \beta \in \Z^d$ and every $T>0$.
Furthermore, for $t \in [0,T]$,
\begin{equation}\label{eq:Phi_ptwise}
	\begin{aligned}
		\bigl|\Phi_{\alpha, \beta}(t)\bigr|
		&\lesssim_T\bk{\frac1{\sqrt{2\overline{c}}}
		\wedge \frac{ t^{{1/2}}}{\Dx}}^{Z(\alpha)}
		t^{-\bk{d + 1}/2} \prod_{j=1}^d
		\bigg(1+\frac{{\big|x_\alpha^j
		-x_\beta^j\big|}^2}
		{2\overline{c}t}\bigg)^{-1}.
	\end{aligned}
\end{equation}
\end{proposition}

\begin{proof}
By applying the bound in \eqref{eq:K_estimate} to $K$,
the convergence of $\Phi_{\alpha, \beta}(t)$---both
pointwise for $(\alpha, \beta) \in \Z^{2d}$ and
in $L^1([0,T])$---along with its pointwise bound, can
now be deduced from the upcoming
Proposition \ref{lem:auxcalc3} presented later.
\end{proof}

\begin{definition}[Fundamental solution]
\label{defin:fundsol_varicoeff}
A fundamental solution of the variable coefficient
problem \eqref{eq:heat_variable_coeff} is a function
$\bm{\Gamma}
\in C^1([0,T];L^\infty_\Dx(\R^d\times\R^d))$
that satisfies
\begin{equation*}
	\begin{cases}\displaystyle
		\frac{\d}{\d t}  \Gamma_{\alpha, \beta}(t)
		=\sum_{j = 1}^dc_\alpha^j
		\nabla_+^j \nabla_-^j \Gamma_{\alpha, \beta}(t),
		& \alpha,\beta\in\Z^d,\ t>0,
		\\ {\Gamma}_{\alpha,\beta}(0)
		=\delta_{\alpha-\beta}, & \alpha,\beta\in\Z^d.
	\end{cases}
\end{equation*}
\end{definition}

The next result establishes the existence of
and pointwise bound on a fundamental solution.

\begin{theorem}[Solution formula]
\label{lem:full_greensfunction}
Let $a_{\alpha, \beta}$ be the fundamental solution
of the equation \eqref{eq:frozencoeff}. For $t \in [0,\infty)$, the
semi-discrete parabolic problem \eqref{eq:heat_variable_coeff}
has a fundamental solution $\bm{\Gamma}$ of the form
\begin{align}\label{eq:Gamma_representation}
	\Gamma_{\alpha, \beta}(t)
	= a_{\alpha - \beta, \beta}(t)
	+ \int_0^t \sum_{\eta \in \Z^d} a_{\alpha - \eta, \eta}(t - s)
	\Phi_{\eta, \beta}(s ) \Dx^d \,\d s,
\end{align}
where $\bm{\Phi}$
is defined in \eqref{eq:Phi_defin}.
Moreover, $\bm{\Gamma}$ possesses the
following properties:
\begin{itemize}
	\item[(i)] For any $\alpha, \beta \in \Z^d$,
	$t \in [0,\infty)$, and $0 \le s \le t$,
	the fundamental solution satisfies
	the propagation relation
	\begin{align}\label{eq:fund_propa}
		\Gamma_{\alpha, \beta}(t)
		= \bigl(\Gamma(s) \dconv
		\Gamma(t-s)\bigr)_{\alpha, \beta}.
	\end{align}

	\item[(ii)] Let $Z(\alpha)$ be the number of zeros
	among the entries of $\alpha \in \Z^d$.
	 For $t \in [0,T]$, the fundamental solution satisfies the bound
	\begin{equation}\label{eq:Gamma-ab-bound}
		\bigl|\Gamma_{\alpha, \beta}(t)\bigr|
		\lesssim_T \bk{\frac1{\sqrt{2 \overline{c}}}
		\wedge \frac{ t^{{1/2}}}{\Dx}}^{Z(\alpha)}
		t^{-d/2} \prod_{j = 1}^d \Biggl(1
		+\frac{{\big|x_\alpha^j -x_\beta^j\big|}^2}
		{2\overline{c} t }\Biggr)^{-1}.
	\end{equation}

	\item[(iii)] For any $p \in [1,\infty]$, $t \in (0,T]$,
	and uniformly in $\alpha,\beta\in\Z^d$, we have
	\begin{equation}\label{eq:Gamma_bound}
		\bigl\|{\Gamma}_{\cdot, \beta}(t)\bigr\|_{L^{p'}(\R^d)}
		+\bigl\|{\Gamma}_{\alpha,\cdot}(t)\bigr\|_{L^{p'}(\R^d)}
		\lesssim_{T,p} t^{-d/(2p)},
	\end{equation}
	where $p'=p/(p - 1)$ is the dual exponent of $p$.
\end{itemize}
\end{theorem}

\begin{proof}
From the sum \eqref{eq:Phi_defin} and its
convergence, we directly have
\begin{align*}
	\Phi_{\alpha, \beta}(t)
	=K_{\alpha, \beta}(t)
	+\int_0^t \sum_{\eta \in \Z^d} K_{\alpha, \eta}(t - s)
	\Phi_{\eta, \beta}(s ) \Dx^d \,\d s.
\end{align*}
From the definition
\eqref{eq:K_defin} of $K_{\alpha,\beta}$, we can readily show
as in \cite[Lemma 4.8]{Fjordholm:2023aa} that
\begin{align*}
	0 & =\frac{\d}{\d t}  \Gamma_{\alpha, \beta}(t)
	-\sum_{j=1}^d c_\alpha^j \nabla_+^j
	\nabla_-^j \Gamma_{\alpha, \beta}(t)\\
	& = \Phi_{\alpha, \beta}(t)
	+ \sum_{j=1}^d c_\beta^j \nabla_+^j
	\nabla_-^j \,a_{\alpha -\beta, \beta}(t)
	- \sum_{j=1}^d c_\alpha^j \nabla_+^j
	\nabla_-^j a_{\alpha - \beta, \beta}(t)\\*
	&\relspace +\int_0^t \sum_{\eta \in \Z^d}
	\sum_{j=1}^d c_\eta^j \nabla_+^j
	\nabla_-^j \,a_{\alpha -\eta, \eta}(t - s)
	\Phi_{\eta, \beta}(s) \Dx^d \,\d s\\
	&\relspace -\int_0^t \sum_{\eta \in \Z^d}
	\sum_{j = 1}^dc_\alpha^j \nabla_+^j \nabla_-^j
	a_{\alpha-\eta,\eta}(t-s)
	\Phi_{\eta,\beta}(s) \Dx^d\,\d s,
\end{align*}
from which \eqref{eq:Gamma_representation} follows.
Equation \eqref{eq:fund_propa} was
shown in \cite[Lemma 4.9]{Fjordholm:2023aa}.

We can now use \eqref{eq:Gamma_representation}
to derive bounds on $\bm{\Gamma}$.
Since $\bm{\Phi}$ satisfies \eqref{eq:Phi_ptwise}, and
$t^{-\hf}\bm{a}(t)$ satisfies the same bounds by
\eqref{eq:pointwise_da} (with $m = 0$), we
can appeal to Proposition \ref{prop:auxcalc1} to get
\begin{align*}
	\sum_{\eta \in \Z^d}&
	a_{\alpha - \eta, \eta}(t - s)
	\Phi_{\eta, \beta}(s)  \Dx^d
	\\ &
	\lesssim_T  \sqrt{t - s}
	\frac{\sqrt{t}}{\sqrt{s \bk{t - s}}}
	\Biggl[\begin{aligned}[t] \biggl(\one{\{t \ge 2 \overline{c} \Dx^2\}}
	+\bk{\frac{ t^{1/2}}{\Dx}}^{Z(\alpha - \beta)}
	\one{\{t < 2 \overline{c} \Dx^2\}}\biggr)&\\
	{}\times	t^{-\bk{d + 1}/2}\prod_{j = 1}^d \Biggl(1
	+\frac{{\big|x_\alpha^j -x_\beta^j\big|}^2}{2\overline{c} t}\Biggr)^{-1}&\Biggr].
	\end{aligned}
\end{align*}
Integrating in $s$, we get:
\begin{align*}
	&\biggl|\int_0^t \sum_{\eta \in \Z^d}
	a_{\alpha - \eta, \eta}(t - s)
	\Phi_{\eta, \beta}(s ) \Dx^d \,\d s\biggr|
	\\ &
	\lesssim_T \Bigg(\one{\{t \ge 2 \overline{c} \Dx^2\}}
	+\bk{\frac{ t^{1/2}}{\Dx}}^{Z(\alpha-\beta)}
	\one{\{t < 2 \overline{c} \Dx^2\}}\Bigg)
	t^{-d/2 + 1/2} \prod_{j = 1}^d
	\bigg(1+\frac{{\big|x_\alpha^j -x_\beta^j\big|}^2}
	{2 \overline{c} t}\bigg)^{-1}.
\end{align*}
Using \eqref{eq:Gamma_representation} and the bound
for the initial term $a_{\alpha - \beta, \beta}(t)$
in \eqref{eq:pointwise_da}, we
conclude that $|\Gamma_{\alpha, \beta}(t)|$
is bounded as asserted in \eqref{eq:Gamma-ab-bound}.

Since  $t^{\hf}/\Dx \lesssim 1$
on the small time regime
$\{4 \overline{c} t < \Dx^2\}$ by definition,
\begin{equation}\label{eq:smalltime_le_largetime}
	\bk{\frac{t^{\hf}}{\Dx}}^{Z(\alpha-\beta)}
	\lesssim 1,
\end{equation}
we have the simpler inequality
\begin{align}\label{eq:Gamma_bound2}
	\bigl|\Gamma_{\alpha, \beta}(t)\bigr|
	&\lesssim_T  t^{-d/2}\prod_{j  = 1}^d \Bigg(1
	+\frac{{\big|x_\alpha^j-x_\beta^j\big|}^2}
	{2 \overline{c} t} \Bigg)^{-1}.
\end{align}
To see \eqref{eq:Gamma_bound}, we take
the $p'$th power on both sides of
\eqref{eq:Gamma-ab-bound} and sum
in $\beta$, giving
\begin{equation} \label{eq:explaining_p'}
	\begin{aligned}
		\bigl\|{\Gamma}_{\alpha,\cdot}(t)\bigr\|_{L^{p'}_\beta}
		 & \lesssim
		 t^{-d/2 + d/(2p')}
		 \Bigg[ t^{-d/2} \sum_{\beta \in \Z^d}
		 \prod_{j = 1}^d
		 \Biggl(1 + \frac{\big|x_\alpha^j
		 -x_\beta^j\big|^2}{2 \overline{c} t }
		 \Biggr)^{-p} \Bigg]^{1/p'}
		 \\ &
		 \lesssim t^{-d/2 + d/(2p')}.
	\end{aligned}
\end{equation}
With $p'=p /(p - 1)$, the exponent of
$t$ works out to be $-d/(2p)$. This establishes
the second part of \eqref{eq:explaining_p'}.
The upper bounding Lorentzian in
\eqref{eq:explaining_p'} is symmetric in the
two indices $\alpha$ and $\beta$,
which gives us the first part
of \eqref{eq:explaining_p'}.
\end{proof}

Next, we proceed to establish bounds on the spatial differences
of $\Gamma_{\alpha, \beta}$, similar to those
for $a_\alpha$ in Proposition \ref{lem:green_estimates}.

\begin{corollary}[Fundamental solution estimates]
\label{lem:full_greensfunction2}
Let $\Gamma_{\alpha, \beta}$ be as given in
Theorem \ref{lem:full_greensfunction}.
For $\ell=1,\dots,d$, recall the definition of
$\nabla_+^\ell$ in \eqref{eq:difference_operators},
and let $(\nabla_+^\ell)^m$ denote $m$
successive applications of $\nabla_+^\ell$.
Then for any $m = 0,1,\ldots$ and $t \in [0,\infty)$, we have
\begin{equation}\label{eq:Gamma-ab-ndiff-est1}
	\begin{aligned}
		\bigl|(\nabla_\pm^\ell)^m \Gamma_{\alpha, \beta}(t)\bigr|
    	&\lesssim \bk{\frac1{\sqrt{2 \overline{c}}}
    	\wedge \frac{ t^{1/2}}{\Dx}}^{Z(\alpha)}
    	\\ &
    	\relspace\times
    	t^{-\bk{d + m}/2}\prod_{j = 1}^d \Biggl(1
    	+\frac{{\big|x_\alpha^j -x_\beta^j\big|}^2}
    	{2 \overline{c}  t }\Biggr)^{-1}.
	\end{aligned}
\end{equation}
Consequently, for any $p\in (d,\infty]$ and $t \in [0,T]$, we have
\begin{equation}\label{eq:Gamma-ab-ndiff-est2}
	\sup_{\beta \in \Z^d}\bigl\|(\nabla_\pm^\ell)^m
	\bm{\Gamma}_{\cdot,\beta}(t)\bigr\|_{L^{p'}_\alpha} +
	\sup_{\alpha \in \Z^d}\bigl\|(\nabla_\pm^\ell)^m
	\bm{\Gamma}_{\alpha,\cdot}(t)\bigr\|_{L^{p'}_\alpha}
	\lesssim t^{d/(2p) - m/2}.
\end{equation}
\end{corollary}

\begin{proof}
Using \eqref{eq:Gamma_representation},
we take the $m$th difference
in direction $\ell$ to get
\begin{align*}
	(\nabla_\pm^\ell)^m\Gamma_{\alpha, \beta}(t)
	&= (\nabla_\pm^\ell)^m a_{\alpha - \beta, \beta}(t)
	\\ &
	\relspace
	+\int_0^t \sum_{\eta \in \Z^d}(\nabla_\pm^\ell)^m
	a_{\alpha-\eta, \eta}(t - s)
	\Phi_{\eta, \beta}(s ) \Dx^d \,\d s.
\end{align*}
Similar to part 2 of the proof of
Theorem \ref{lem:full_greensfunction},
by controlling $|\bk{\nabla_\pm^\ell}^m a_{\alpha,\beta}(t)|$
using \eqref{eq:pointwise_da} and managing $\Phi$
through \eqref{eq:Phi_ptwise}, we apply
the upcoming Lemma \ref{lem:auxcalc2}.
This approach allows us to once again overestimate
the small time portion of the estimate as
demonstrated in \eqref{eq:smalltime_le_largetime},
thereby obtaining
\begin{align*}
	\biggl|\int_\tau^t \sum_{\eta \in \Z^d}
	(\nabla_\pm^\ell)^m a_{\alpha - \eta, \eta}(t - s)
	\Phi_{\eta, \beta}(s ) \Dx^d \,\d s\biggr|
	\lesssim t^{- (d + m)/2}
	\prod_{j = 1}^d
	\Biggl(1 + \frac{{\big|x_\alpha^j -x_\beta^j\big|}^2}
	{2 \overline{c}t}\Biggr)^{-1},
\end{align*}
from which \eqref{eq:Gamma-ab-ndiff-est1} follows.
Calculating as in \eqref{eq:explaining_p'}, the
$L^{p’}$ control given in \eqref{eq:Gamma-ab-ndiff-est2}
directly follows from the pointwise
bound \eqref{eq:Gamma-ab-ndiff-est1}.
\end{proof}

Similar to the Duhamel representation
formula established in Lemma \ref{lem:duhamel_representation},
we can derive a solution formula for
\eqref{eq:heat_variable_coeff}
using a fixed point argument.
A detailed proof can be deduced from
\cite[Lemmas 4.11, 4.12]{Fjordholm:2023aa}.

\begin{lemma}\label{lem:duhamel_representation2}
Fix $p,q$ satisfying $p>d$ and $1/q+d/(2p)<1$.
\begin{itemize}
	\item[(i)] Let $\bm{f} \in L^{q}\bigl([0,T];L^p_\Dx(\R^d)\bigr)$
	and $\bm{\psi}  \in L^\infty_\Dx(\R^d)$.
	Then the solution $\bm{ u}$ of the
	non-homogeneous semi-discrete parabolic
	equation \eqref{eq:heat_variable_coeff} is given by
	\begin{equation*}
		u_\alpha(t)
		=\sum_{\eta \in \Z^d} \Gamma_{\alpha, \eta}(t) \psi_\eta \Dx^d
		+\int_0^t \sum_{\eta \in \Z^d}
		\Gamma_{\alpha, \eta}(t - s) f_{\eta}(s) \Dx^d \,\d s.
	\end{equation*}

	\item[(ii)]
	Let ${\bm Y} \in L^p_\Dx(\R^d)$ and ${\bm f} \in
	L^p_\Dx(\R^d) + L^\infty_\Dx(\R^d)$.
	Furthermore, suppose ${\bm \psi} \in L^\infty_\Dx(\R^d)$
	with $\nabla_+ \nabla_- {\bm \psi} \in L^\infty_\Dx(\R^d)$.
	There exists a unique solution ${\bm u} \in C^1([0,T];L^\infty_\Dx(\R^d))$
	to the Cauchy problem
	\begin{equation*}
		\begin{cases}
			{\displaystyle \frac{\d}{\d t}  u_\alpha
			-\sum_{j = 1}^d c_\alpha^j \nabla_+^j \nabla_-^j u_\alpha
			+Y_\alpha u_\alpha = f_\alpha,} &\alpha\in\Z^d,\ t>0,
			\\
			{\bm u}(0) = {\bm \psi}.
		\end{cases}
	\end{equation*}
	The solution $\bm{u}$ is non-negative
	on $[0,\infty)\times\R^d$ and satisfies
	\begin{equation}\label{eq:phi-bounds}
		\|\bm{u}\|_{L^\infty([0,T]\times \R^d)},
		\quad \|\nabla_\pm \bm{ u}\|_{L^\infty([0,T]\times \R^d)}
		\lesssim 1,
	\end{equation}
	where the implicit constant depends only on the values
	$T$, ${\bm c}$,  $\|{\bm Y}\|_{L^p(\R^d)}$,
	and $\norm{\bm f}_{L^p(\R^d) + L^\infty(\R^d)}$---in
	particular, there is no dependence on
	$\norm{{\bm Y}}_{L^\infty(\R^d)}$.
\end{itemize}
\end{lemma}

\begin{proof}[Sketch of proof]
The first bound on \eqref{eq:phi-bounds} can be
derived as a by-product of a fixed point argument
on $L^\infty([0,h]\times \T^d)$ for sufficiently
small $h > 0$ that does not depend on the initial condition,
and gives us a Duhamel representation formula
for the solution:
\begin{equation}\label{eq:duhamel_representation2}
	\begin{split}
		u_\alpha(t)
		=~&\sum_{\eta \in \Z^d} \Gamma_{\alpha, \eta}(t)
		\psi_\eta \Dx^d
		+\int_0^t \sum_{\eta \in \Z^d}
		\Gamma_{\alpha, \eta}(t - s) Y_{\eta}(s)
		u_\eta(s) \Dx^d \,\d s
		\\ &+\int_0^t \sum_{\eta \in \Z^d}
		\Gamma_{\alpha, \eta}(t-s) f_{\eta} \Dx^d \,\d s.
	\end{split}
\end{equation}
From \eqref{eq:duhamel_representation2},
the first bound \eqref{eq:phi-bounds}
follows by Gronwall's inequality and the
discrete Young inequality (Lemma \ref{lem:young}):
\begin{align*}
\int_0^t \sum_{\eta \in \Z^d}
		\Gamma_{\alpha, \eta}(t - s) Y_{\eta}(s)
		u_\eta(s) \Dx^d \,\d s &\leq \int_0^t \sup_{\alpha \in \Z^d} \|\Gamma_{\alpha, \cdot}(t - s)\|_{L^{p'}_\eta}
        \|{\bm Y}\|_{L^p_x}\|{\bm u}(s)\|_{L^\infty_x}\,\d s,\\
\int_0^t \sum_{\eta \in \Z^d}
		\Gamma_{\alpha, \eta}(t-s) f_{\eta} \Dx^d \,\d s
  &\le \sup_{\alpha \in \Z^d} \|\Gamma_{\alpha, \cdot}\|_{L^1_t L^{p'}_\eta}
    \|{\bm f}\|_{L^p_x}.
\end{align*}

To get the gradient bound, we consider the equation
for $\tilde{\bm u} \coloneqq {\bm u} - {\bm \psi}$. This function
satisfies
\begin{align*}
	&\frac{\d }{\d t} \tilde{u}_\alpha
	- \sum_{j = 1}^d c_\alpha^j
	\nabla_+^j\nabla_-^j \tilde{u}_\alpha
	+ Y_\alpha\tilde{u}_\alpha + \tilde{f}_\alpha = 0,
	\qquad \tilde{\bm u}(0) \equiv 0, \\
\intertext{where}
	&\tilde{f}_\alpha
	\coloneqq\sum_{j=1}^d c_\alpha^j \nabla_+^j\nabla_-^j
	\psi_\alpha+Y_\alpha \psi_\alpha,
\end{align*}
so that $\tilde{\bm f} \in L^\infty_\Dx(\R^d)
+L^p_\Dx(\R^d)$.
Therefore \eqref{eq:duhamel_representation2}
holds with $\tilde{\bm u}$ in place of ${\bm u}$
and $\tilde{\bm f}$ in place of ${\bm f}$.
Taking a difference in
\eqref{eq:duhamel_representation2}, we then get
\begin{align*}
	\norm{\nabla_+ \tilde{\bm u}}_{L^\infty_{t,x}}
	& \lesssim_T
	\norm{\nabla_+{\bm \Gamma}}_{L^1_t
	L^\infty_\alpha L^{p'}_\beta}
	\norm{\bm Y}_{L^p_x}\norm{\tilde{\bm u}}_{L^\infty_{t,x}}
	\\ & \qquad
	+\norm{\nabla_+\bm\Gamma}_{L^1_tL^\infty_\alpha
	L^1_\beta}
	\|\nabla_+ \nabla_-{\bm \psi}\|_{L^\infty_x}
	+\norm{\nabla_+\bm\Gamma}_{L^1_tL^\infty_\alpha
	 L^{p'}_\beta}
	\|{\bm Y}\|_{L^p_x}\|{\bm \psi}\|_{L^\infty_x}.
\end{align*}
By incorporating the $L^\infty_{t,x}$ bound
for $\tilde{\bm u}$ from \eqref{eq:phi-bounds},
together with the estimates for
$\nabla_+ \bm \Gamma$ in the spaces
$L^\infty_\alpha L^{p’}_\beta$ and
$L^\infty_\alpha (L^1_\beta \cap L^{p’}_\beta)$---expressed
in terms of $t^{d/(2p)-1/2}$ and $t^{-1/2}$, both of which
are integrable on $[0,T]$ since $p > d$---the second
inequality in \eqref{eq:phi-bounds} follows. These time
decay factors, $t^{d/(2p)-1/2}$ and $t^{-1/2}$, originate
from \eqref{eq:Gamma-ab-ndiff-est2}  with $m = 1$,
where the $L^1_\beta$ bound arises by taking
$p = \infty$ ($p'=1$),
yielding the decay rate $t^{-1/2}$.
\end{proof}

\section{Iterated convolutions}
\label{sec:aux_calc}

In this section,
we come to the heart of the matter,
which is the correct handling of the iterated
temporal and discrete spatial convolutions in the definition
\eqref{eq:Km_defin} of $K^{(m)}_{\alpha, \beta}$.
We provide the proof of Theorem \ref{thm:Phi_bound} 
concerning the properties of the fundamental object
$\bm{\Phi}$ defined in \eqref{eq:Phi_defin}.
This completes the construction of the fundamental
solution for the variable-coefficient, semi-discrete 
heat operator $u_\alpha \mapsto
\frac{\d}{\d t}u_\alpha
- \sum_{j=1}^d c_\alpha^j
\nabla_+^j\nabla_-^j  u_\alpha$,
see Definition \ref{defin:fundsol_varicoeff}.

Our strategy is as follows: We consider the convolution
$\bk{K(s)\dconv K^{(m - 1)}(t - s)}_{\alpha, \beta}$
at fixed times $s$ and $t - s$.
In Lemma \ref{thm:cont_convolve}, we compute the
continuous analogue of known bounds of the spatial
sum at fixed times exactly. In Lemma \ref{thm:auxcalc4},
we show how the discrete spatial convolution is bounded
by the continuous spatial convolution.
All these feed into  the key inductive
step in Proposition \ref{prop:auxcalc1}, which provide
the estimates showing how to iterate the discrete spatial sum.
In Lemma \ref{lem:auxcalc2}, we carry out the
induction on $m$ explicitly along with the temporal
convolution. Finally in Proposition \ref{lem:auxcalc3},
we wrap up the calculations by showing that
the sum  $\Phi_{\alpha, \beta}$ of
$K^{(m)}_{\alpha, \beta}$ over $m$
required in \eqref{eq:Phi_defin} converges
appropriately and satisfies the requisite pointwise
bound \eqref{eq:Phi_ptwise}.

Throughout we shall use the notation:
\begin{align*}
	\tilde{L}(t,z) \coloneqq \frac1{\sqrt{t}} \bk{1 + \frac{z^2}t}^{-1}.
\end{align*}
The symbol ``$\Gamma$" exclusively
denotes the Gamma function in this section.

\begin{lemma}[Convolution of Lorentzians]\label{thm:cont_convolve}
For any $0 \le s \le t $, $x,y \in \R$, let
\begin{align}\label{eq:integral_xt_convolve}
I(x,y,s,t) &\coloneqq\int_{\R} \tilde{L}(t - s, x - z) \tilde{L}(s, z-y)\,\d z.
\end{align}
Then
\begin{align*}
	I(x,y,s,t) =
	 \pi
	\tilde{L}\Bigl(\bk{\sqrt{t - s} + \sqrt{s}}^2\!,\, x - y\Bigr),
\end{align*}
and
\begin{align}\label{eq:Lorentz_eval}
	I(x,y,s,t)
	\le \sqrt2 \pi \tilde{L}(t, x - y).
\end{align}
\end{lemma}

\begin{proof}
Denote the spatial Fourier
transform on $\R$ by $\mathcal{F}$  (specifically, with 
$(\mathcal{F} f)(\xi) := \int_\R f(x) e^{-2\pi i x\xi}\,\d x$)
and its inverse by $\mathcal{F}^{-1}$.
The Fourier transform of the convolution
is the product of the Fourier transforms on $L^2(\R)$.
Using the Fourier transform we find:
\begin{equation*}
	\begin{aligned}
		I(x,y,s,t)& = \int_{\R} \tilde{L}(t - s, x - y - w) \tilde{L}(s,w)\,\d w
		\\ &
		=\mathcal{F}^{-1}\Bigl(\mathcal{F}\tilde{L}(t-s, \cdot)
		\mathcal{F} \tilde{L}(s, \cdot) \Bigr)(x-y)\\
		&= \mathcal{F}^{-1}\Bigl(\pi^2
	e^{- 2\pi(\sqrt{t - s} + \sqrt{s})\abs{\cdot}}\Bigr)(x - y)\\
	&=\mathcal{F}^{-1}\Bigl( \pi
	\mathcal{F} \tilde{L}\Bigl(\bigl(\sqrt{s}
	+ \sqrt{t - s}\bigr)^2,\cdot\Bigr)(\cdot)\Bigr)(x - y).
	\end{aligned}
\end{equation*}
We used that the Fourier transform of
the Lorentzian $\tilde L$ is
$\mathcal{F}\tilde{L}(t,\cdot)(\xi)
=\pi e^{- 2\pi\sqrt{t}\abs{\xi}}$. 
(This can be seen by taking the inverse transform:
$\tilde L(t,x)=\pi\int_\R e^{- 2\pi\sqrt{t}\abs{\xi}}
e^{2\pi i x\xi}\,\d\xi$.)
Noting that the map $\tau\mapsto\frac1{\sqrt{\tau}} \tilde{L}(\tau,z)$
is decreasing, and that $\sqrt{t - s}+\sqrt{s} \geq \sqrt{t}
\geq \frac1{\sqrt2}\bigl( \sqrt{t - s} + \sqrt{s}\bigr)$, we obtain
$$
I(x,y,s,t) \le \pi\frac{
\sqrt{s}+\sqrt{t-s}}{\sqrt{t}} \tilde{L}(t, x - y)
\le \sqrt2\pi \tilde{L}(t, x - y).
$$
\end{proof}

\begin{lemma}\label{thm:auxcalc4}
Fix any $0 < s < t\le T$,
$\alpha, \beta \in \Z$, and set
$$
g(z) \coloneqq \tilde{L}(t - s, x_\alpha - z) \tilde{L}(s, x_\beta - z),
$$
where $x_\alpha = \alpha \Dx$, $x_\beta = \beta \Dx$.
Let $I$ be as
in \eqref{eq:integral_xt_convolve}.
Then the following bound holds:
\begin{align}\label{eq:trunc_convolve}
	\sum_{\eta \in \Z\backslash \{\alpha, \beta\}} g(x_\eta) \Dx
	\le e^8  I(x_\alpha, x_\beta,s,t).
\end{align}\\
If $\alpha = \beta$ and $0 \le s, t - s
\le  C \Dx^2$ for some $C>0$, then
\begin{align}\label{eq:trunc_convolve4}
	\sum_{\eta \in \Z\backslash \{\alpha\}}g(x_\eta)\Dx
	\le  \frac{3C^2}\Dx.
\end{align}
And if $t-s, s\ge\delta$
for some $\delta > 0$, then
\begin{align}\label{eq:trunc_convolve3}
	\sum_{\eta \in \Z} g(x_\eta) \Dx
	\le \bigg[ e^8 +  \bigg(1 + \frac{\Dx^2}
	{\delta}\bigg)^2\bigg]
	I(x_\alpha, x_\beta, s,t).
\end{align}
\end{lemma}

\begin{proof}
{\em 1. Proof of \eqref{eq:trunc_convolve}.}

Differentiating $g$ we find:
\begin{align*}
	g'(z) = 2 g(z)\underbrace{ \bk{
	\frac{\bk{x_{\alpha} - z}/\bk{t - s}}
	{1 + \frac{\bk{x_{\alpha} - z}^2}{t - s}}
	+\frac{\bk{x_{\beta} - z}/s}
	{1 + \frac{\bk{x_{\beta}-z}^2}{s}}}}_{\eqqcolon\, U(z)}.
\end{align*}
For any $z_1, z_2 \in \R$, we have
\begin{align}\label{eq:g_log_estm}
	\log\biggl(\frac{g(z_2)}{g(z_1)}\biggr)
	=2\int_{z_1}^{z_2} U(z)\,\d z.
\end{align}
If $\eta \neq \alpha, \beta$, we have
$\abs{x_\alpha - z}, \abs{x_\beta - z} \ge \Dx/2$
for $z \in \cell_\eta\coloneqq [x_{\eta}
-\Dx/2 , x_{\eta} + \Dx/2]$.
Therefore,
\begin{align*}
	\max_{z\in\cell_\eta} \abs{U(z) }
	\le \frac1{\abs{x_\alpha - z}\wedge\abs{x_\beta - z}}
	\bk{\frac{\abs{x_{\alpha} - z}^2/\bk{t - s}}
	{1+\frac{\bk{x_{\alpha} - z}^2}{t - s}}
	+\frac{\abs{x_{\beta}-z}^2/s}
	{1+\frac{\bk{x_{\beta}-z}^2}{s}}} \le \frac4\Dx.
\end{align*}
Inserting this back into \eqref{eq:g_log_estm},
using $\abs{z_1 - z_2} \le \Dx$ for
$z_1, z_2 \in \cell_\eta$, we get
\begin{align*}
	\frac{\max_{z \in \cell_\eta}
	g(z)}{\min_{x \in \cell_\eta} g(z)} \le e^8.
\end{align*}
Therefore, for any $t > s > 0$,
\begin{align*}
	g(x_{\eta})  \le \max_{z \in \cell_\eta} g(z)
	\le e^8 \min_{z \in \cell_\eta} g(z)
	\le \frac{e^8}\Dx \int_{x_{\eta}-\Dx/2}^{x_{\eta}
	+ \Dx/2} g(z) \,\d z.
\end{align*}
Using $g(z) \ge 0$, we can sum up in
$\eta \in \Z\backslash\{\alpha, \beta\}$ to conclude
that
\begin{align}\label{eq:g_sum_conclude1}
	\sum_{\eta \in \Z\backslash \{\alpha, \beta\}} g(x_\eta) \Dx
	\le e^8 \sum_{\eta \in \Z\backslash\{\alpha, \beta\}}
	\int_{x_{\eta} - \Dx/2}^{x_{\eta} + \Dx/2} g(z) \,\d z
	\le e^8 I(x_\alpha, x_\beta, s,t).
\end{align}

\smallskip
\noindent {\em 2. Proof of \eqref{eq:trunc_convolve4}.}

We turn to the  bound for $g$
when $\alpha = \beta$.
By the elementary estimate $(1 + \abs{v})^{-1}
\le\abs{v}^{-1}$, we have directly:
\begin{align*}
	\sum_{\eta \neq \alpha} g(x_\eta) \Dx
	&\le \sum_{\eta \neq \alpha}
	\frac{\sqrt{s \bk{t-s}}}{\abs{\alpha-\eta}^4} \Dx^{-3}
	\le \frac{\pi^4}{90} \frac{C^2}{\Dx},
\end{align*}
where we used the assumption
$0 \le s, t - s \le C \Dx^2$ and the fact
$\sum_{k \ge 1} k^{-4} = \pi^4/90$
in the second inequality.

\smallskip
\noindent {\em 3. Proof of \eqref{eq:trunc_convolve3}.}

For $t - s , s \ge \delta$,
on $\cell_\alpha$,
\begin{align*}
	g(z) \ge \frac{ \bk{1 + \Dx^2/\delta}^{-2}}{\sqrt{s(t - s)}}.
\end{align*}
On the other hand, we always have $g(z) \le 1/\sqrt{s(t - s)}$.
Therefore,
\begin{align*}
	\int_{x_{\alpha} - \Dx/2}^{x_{\alpha} + \Dx/2}
	g(z) \,\d z  \ge \Dx \bk{1 + \frac{\Dx^2}{\delta}}^{-2}
    \frac1{\sqrt{s (t - s)}}
	\ge \Dx \bk{1 + \frac{\Dx^2}{\delta}}^{-2} g(x_\alpha) .
\end{align*}
Inserting this back
into \eqref{eq:g_sum_conclude1}, we have
\begin{align*}
	\sum_{\eta \in \Z} g(x_\eta) \Dx
	&\le e^8 \sum_{\eta \in \Z\backslash \{\alpha, \beta\}}
	\int_{x_{\eta} - \Dx/2}^{x_{\eta} + \Dx/2}g(z) \,\d z
	\\ & \quad\,\,
	+ \bk{1 + \frac{\Dx^2}{\delta}}^{2}
	\bk{\int_{x_{\alpha} - \Dx/2}^{x_{\alpha} + \Dx/2}g(z) \,\d z
	+\int_{x_{\beta} - \Dx/2}^{x_{\beta} + \Dx/2}g(z) \,\d z}
	\\ &
	\le \bigg[e^8+\bigg(1+\frac{\Dx^2}{\delta}\bigg)^{2}\bigg]
	I(x_\alpha, x_\beta, s,t).
\end{align*}
\end{proof}

\begin{proposition}\label{prop:auxcalc1}
Let $C_1 > 0$, $t\geq0$, $\alpha,\beta\in\Z^d$.
Let $Z(\alpha)$ be the number of zeros among the
entries of the $d$-tuple $\alpha$.  Set
\begin{equation}\label{eq:K_ab_parts}
	f(t,\alpha) \coloneqq \bk{1 \wedge
	\frac{C_1 t}{\Dx^2}}^{Z(\alpha)/2}t^{-1/2}
	\prod_{j=1}^d \tilde{L}\bigl(C_1t,\alpha^j \Dx\bigr).
\end{equation}
Then for any $0<s<t$,
we have the following estimates
for the discrete convolutions:
\begin{align}\label{eq:convolution-estimate}
	\sum_{\eta \in \Z^d}
	f(t-s,\alpha-\eta) f(s,\eta-\beta)\Dx^d
	\lesssim \frac{\sqrt{t}f(t,\alpha-\beta)}
	{\sqrt{s\bk{t-s}}},
\end{align}
where the implicit constant depends
only on $C_1$ and $d$, and specifically,
is independent of $\Dx$ and $T$.
\end{proposition}

\begin{remark}
Splitting $f(t, \alpha) $ into $f_<(t, \alpha, 1) + f_>(t, \alpha, 1)$ where
\begin{align*}
f_>(t, \alpha, \lambda)
	&\coloneqq  t^{-1/2} \one{\{C_1t \ge \Dx^2/\lambda\}}
	  \prod_{j = 1}^d
 	\tilde{L}\bigl(C_1t,\alpha^j \Dx \bigr),\\
f_<(t,\alpha,\lambda)
	&\coloneqq\bk{\frac{\sqrt{C_1}t^{1/2}}{\Dx}}^{Z(\alpha - \beta)}
    \one{\{ C_1t < \Dx^2/\lambda\}}
    t^{- 1/2} 	\prod_{j = 1}^d
	\tilde{L}\bigl(C_1 t, \alpha^j \Dx\bigr),
\end{align*}
and $f_>(t, \alpha) := f(t,\alpha,1)$, 
it is also possible to show the \emph{interaction estimates}
\begin{align*}
& \sum_{\eta \in \Z^d}
	f_<(t - s,\alpha - \eta) f_< (s, \eta - \beta)\Dx^d
\lesssim \frac{\sqrt{t}f_< (t,\alpha - \beta,\hf )}
	{ \sqrt{s\bk{t - s}}},\\
&
\sum_{\eta \in \Z^d} f_>(t - s,\alpha -  \eta)
f_< (s, \eta - \beta) \Dx^d
\lesssim  \frac{\sqrt{t}f_>(t,\alpha - \beta, \hf)}{\sqrt{s\bk{t - s}}} 	
+ \frac{\sqrt{t}f_<(t,\alpha - \beta, \hf)}{\sqrt{s\bk{t - s}}},
	\\
&\sum_{\eta \in \Z^d} f_>(t - s,\alpha - \eta)
f_> (s, \eta - \beta) \Dx^d
\lesssim \frac{\sqrt{t}f_>(t,\alpha - \beta,\hf )}
	{\sqrt{s\bk{t - s}}}.
\end{align*}
\end{remark}

\begin{proof}[Proof of Proposition \ref{prop:auxcalc1}]
By scaling $t\mapsto t/C_1$ (whereby $f$ picks
up a constant multiplicative factor $C_1^{(d+1)/2}$), we can assume
without loss of generality that $C_1 = 1$.
For $t \in [0,\infty)$ and $k \in \Z$, set
\begin{align*}
	h(t,k) \coloneqq \bk{\one{k\neq0}
	+\biggl(1 \wedge \frac{t^{1/2}}{\Dx}\biggr)
	\one{k = 0}} \tilde{L}(t,k\Dx).
\end{align*}
We first factorise the sum
\eqref{eq:convolution-estimate} to get:
\begin{align*}
	&\sum_{\eta\in\Z^d} f(t-s,\alpha-\eta)
	f(s,\eta - \beta)\,\Dx^d
	\\*
	&\qquad =\frac{1}{\sqrt{s(t-s)}}\prod_{j=1}^d
	\sum_{k\in\Z}h(t-s,\alpha^j-k)h(s,k -\beta^j)\,\Dx.
\end{align*}
In order to arrive
at \eqref{eq:convolution-estimate},
it suffices to show that
\begin{equation}\label{eq:convolution-estimate-j}
	A^j \coloneqq \sum_{k\in\Z} h(t-s,\alpha^j-k)
	h(s,k-\beta^j)\,\Dx
	\lesssim h(t,\alpha^j-\beta^j)
	\qquad
	\forall\ j=1,\dots,d.
\end{equation}

We decompose $A^j$ into
\begin{align*}
	A^j &=
	\begin{aligned}[t]
		&\sum_{k\neq\alpha^j,\beta^j}
		h\bigl(t-s,(\alpha^j-k)\Dx\bigr)
		h\bigl(s,(k-\beta^j)\Dx\bigr)\Dx
		\\ &
		+\ind_{\alpha^j=\beta^j}h(t-s,0)h(s,0)\Dx
		\\ &
		+\ind_{\alpha^j\neq\beta^j}
		\Bigl(h(t-s,0)h(s,\alpha^j-\beta^j)
		+h(t-s,\alpha^j- \beta^j)h(s,0)\Bigr)\Dx
	\end{aligned}
	\\ & \eqqcolon
	A_1^j + A_2^j + A_3^j.
\end{align*}
We claim that each of these terms can be
bounded (up to some absolute constant) by
$h(t,\alpha^j,\beta^j)$. We subsequently
abbreviate using $z \coloneqq
x_\alpha^j-x_\beta^j = (\alpha^j-\beta^j)\Dx$.

By Lemma 4.1 and 4.2, we have
the two estimates
$$
A_1^j
\leq e^8 I(x_\alpha^j,x_\beta^j,s,t)
= 2e^8 \tilde{L}(t,z)
$$
and (when $\alpha^j=\beta^j$)
$$
A_1^j = C\frac{s(t-s)}{\Dx^3}
$$
for some absolute constant $C>0$. Hence,
if both $\alpha^j=\beta^j$ and $t<\Dx^2$ we get
$$
A_1^j=C\frac{s(t-s)}{\Dx^3} \leq C\frac{t^2}{\Dx^3}
\leq C\frac{1}{\Dx}=Ch(t,\alpha^j-\beta^j),
$$
while if either $\alpha^j\neq\beta^j$ or
$t\geq \Dx^2$ (or both) we get
$A_1^j\leq C\tilde{L}(t,z)
=Ch(t,\alpha^j-\beta^j)$.

Next, by direct insertion, we get
(assuming $\alpha^j=\beta^j$)
$$
A_2^j = \frac{\Dx}{\sqrt{s(t-s)}}
\biggl(1\wedge\frac{\bk{t-s}^{{1/2}}}{\Dx}\biggr)
\biggl(1\wedge\frac{s^{1/2}}{\Dx}\biggr),
$$
and a case-by-case analysis (over whether
$s$ or $t-s$ are greater or smaller
than $\Dx^2$) reveals that
$$
A_2^j \leq \sqrt{2}\frac{1}{\sqrt t}
=\sqrt{2 } h(t,0).
$$

Finally, for $A_3^j$ we get (assuming $\alpha^j\neq\beta^j$,
and setting $z=(\alpha^j-\beta^j)\Dx$)
\begin{align*}
	A_3^j
	&=\biggl(\tilde{L}(t-s,0)\tilde{L}(s,z)
	\biggl(1\wedge\frac{\sqrt{t-s}}{\Dx}\biggr)
	+ \tilde{L}(t-s,z)\tilde{L}(s,0)
	\biggl(1\wedge\frac{\sqrt{s}}{\Dx}\biggr)\biggr)\Dx
	\\ &
	=\tilde{L}(s,z)
	\biggl(1\wedge\frac{\Dx}{\sqrt{t-s}}\biggr)
	+\tilde{L}(t-s,z)
    \biggl(1\wedge\frac{\Dx}{\sqrt{s}}\biggr).
\end{align*}
We shall see that both of these terms can be bounded by
$2h(t,\alpha^j - \beta^j)$; since the
two terms are symmetric, we only consider
the first one, and name this $A_{3,1}^j$.

Before proceeding we record the following
simple observations:
\begin{lemma}\label{lem:lorentzian-behavior}
Fix $z\neq 0$.
\begin{enumerate}
	\item[(i)] The function $s\mapsto \sqrt{s}\tilde{L}(s,z)$
	vanishes at $s=0$ and is increasing for $s>0$.

	\item[(ii)] The function $\tilde{L}(s,z)$ (of $s\geq0$)
	vanishes at $s=0$, has a single maximum at
	$s^*\coloneqq z^2$, and converges
	to zero as $s\to\infty$.
\end{enumerate}
\end{lemma}

Recall that $z\coloneqq(\alpha^j-\beta^j)\Dx$
is a non-zero integer multiple of $\Dx$.
We split the proof that $A_{3,1}^j
\lesssim h(t,\alpha^j - \beta^j)=\tilde{L}(t,z)$
(assuming $\alpha^j\neq \beta^j$)
into several cases.

\begin{enumerate}[label=(\alph*)]
	\item \underline{$t\leq z^2$:}
	Lemma \ref{lem:lorentzian-behavior}(ii) yields
	$A_{3,1}\leq
    \tilde L(s) \leq \tilde L(t)$,

	\item \underline{$z^2<t\leq 2\Dx^2$:} (Necessarily, $z=\pm\Dx$.)
	By Lemma \ref{lem:lorentzian-behavior}(ii) we get again
	$$
	A_{3,1} \leq \tilde{L}(s,z) \leq \tilde{L}(\Dx^2,z)
	=\tilde{L}(t,z){\frac{\tilde{L}(\Dx^2,z)}
	{\tilde{L}(t,z)}}
	\leq \sqrt{2}\tilde{L}(t,z),
	$$
	where we have used $\frac{\tilde{L}(\Dx^2,z)}{\tilde{L}(t,z)}
	\leq \frac{\tilde{L}(\Dx^2,z)}{\tilde{L}(2\Dx^2,z)} \le \sqrt{2}$.

	\item \underline{$t> 2\Dx^2 \vee z^2$:} We
	split into two sub-cases:
	\begin{enumerate}[label=(c$_{\arabic*}$)]
		\item \underline{$t-s \leq \Dx^2$:}
		Using first Lemma \ref{lem:lorentzian-behavior}(i)
		and then the fact that $s\geq t-\Dx^2 \geq t/2$, we get
		$$
		A_{3,1}^j = \tilde{L}(s,z)
		\leq \tilde{L}(t,z)
		\underbrace{\frac{t^\hf}{s^\hf}}_{\leq\,\sqrt2}
		\leq \sqrt2 \tilde{L}(t,z).
		$$

		\item \underline{$t-s>\Dx^2$:} Set $s_0 \coloneqq s \vee \Dx^2$.
		Using Lemma \ref{lem:lorentzian-behavior}(i) and (ii), we get
		\begin{align*}
			A_{3,1}^j &= \tilde{L}(s,z)\frac{\Dx}{\sqrt{t-s}}
			\leq \tilde L(s_0,z)\frac{\Dx}{\sqrt{t-s_0}}
			=\tilde L(t,z) \frac{\tilde{L}(s_0,z)}{\tilde{L}(t,z)}\frac{ \Dx}
			{ \sqrt{t-s_0}}
			\\ &
			=\tilde L(t,z) \underbrace{\frac{\sqrt{s_0}\tilde{L}(s_0,z)}{\sqrt{t}\tilde{L}(t,z)}}_{\leq\,1}
			\underbrace{\biggl(\frac{\sqrt t}{\sqrt s_0}\wedge
			\frac{\sqrt t}{\sqrt{t-s_0}}\biggr)}_{\leq\,\sqrt 2}
			\underbrace{\biggl(\frac{\Dx}{\sqrt s_0}\vee
			\frac{\Dx}{\sqrt{t-s_0}}\biggr)}_{\leq\,1}
			\leq \sqrt 2 \tilde L(t,z).
		\end{align*}
	\end{enumerate}
\end{enumerate}

Taken together, we arrive at
$$
A^j \leq Ch(t,\alpha^j-\beta^j)+\ind_{\alpha^j
=\beta^j} h(t,\alpha^j-\beta^j)+\sqrt{2}
\ind_{\alpha^j\neq \beta^j}h(t,\alpha^j-\beta^j),
$$
which is precisely \eqref{eq:convolution-estimate-j}.
\end{proof}

\begin{lemma}[Higher auxiliary
kernel estimates]\label{lem:auxcalc2}
Let $f$ be as defined in \eqref{eq:K_ab_parts},
for some $C_1>0$. Assume $K^{(1)}=K^{(1)}_{\alpha,\beta}(t)$
satisfies
\begin{align*}
	\bigl|K^{(1)}_{\alpha, \beta}(t)\bigr|
	&\le C f(t,\alpha-\beta).
\end{align*}
for a constant $C > 0$ and
all $0 \le s \le t \le T$. Define $K^{(2)},
K^{(3)},\dots$ iteratively by
$$
K^{(m)}_{\alpha, \beta}(t) \coloneqq
\int_0^t  \Big(K(t - s) \dconv
K^{(m - 1)}(s)\Big)_{\alpha, \beta}\,\d s
$$
Then for some absolute
constant $C_3 > 0$, $K^{(m)}$ satisfies
\begin{align}\label{eq:Km_ab_bound}
	\bigl|K^{(m)}_{\alpha, \beta}(t)\bigr|
	\le \frac{C C_3^m}{\Gamma(m/2)}
	t^{(m - 1)/2} f(t,\alpha-\beta)
\end{align}
for all $t\in[0,T]$, $\alpha,\beta\in\Z^d$,
and $m\in\N$.
\end{lemma}

\begin{proof}
For $m=1$ the right-hand side of \eqref{eq:Km_ab_bound}
is $\frac{CC_3}{\Gamma(1/2)} f(t,\alpha-\beta)$,
so the claim holds as long as $C_3\geq\frac{\Gamma(1/2)}{C}$.
We now induct on $m$. Proposition \ref{prop:auxcalc1}
provides that for some constant $C' > 0$,
\begin{align*}
	&\int_0^t \Big(K(t - s) \dconv
	K^{(m - 1)}(s)\Big)_{\alpha, \beta}\,\d s
	\\ &\qquad
	\le \frac{CC_3^{m-1}}{\Gamma(\bk{m - 1}/2)}
	\int_0^t s^{\bk{m-2}/2} \sum_{\eta \in \Z^d}
	f(t-s,\alpha-\eta)f(s,\eta-\beta) \,\d s
	\\ &\qquad
	\le  \frac{CC_3^{m-1}C'}{\Gamma(\bk{m-1}/2)}
	\int_0^t s^{\bk{m-2}/2}
	\frac{\sqrt{t}}{\sqrt{s\bk{t - s}}}\,\d s
	f(t,\alpha,\beta).
\end{align*}
Using the following fact
about beta functions \cite[Theorem 2.1.2]{BW2016}
$$
\forall a,b \in \mathbb{C}:
\mathfrak{R} a,\mathfrak{R} b>-1,
\qquad
\int_0^1 (1-v)^{a} v^{b}\,\d v
=\frac{\Gamma(a+1) \Gamma(b+1)}{\Gamma(a + b + 2)},
$$
we obtain
\begin{align*}
	&\int_0^t  \Big(
	K(t - s) \dconv K^{(m - 1)}(t - s)\Big)_{\alpha, \beta} \,\d s
	\\ & \qquad
	\le \frac{CC_3^{m-1}C'}{\Gamma(\bk{m - 1}/2)}
	\frac{\Gamma(\bk{m - 1}/2) \Gamma(\hf)}{\Gamma(m /2)}
	t^{\bk{m - 1}/2}f(t,\alpha-\beta)
	\\ & \qquad
	=\frac{CC_3^{m-1}}{\Gamma(m/2)} C'\Gamma(\hf)
	t^{\bk{m-1}/2}f(t,\alpha-\beta),
\end{align*}
which proves the lemma
with $C_3\geq C'\Gamma(\hf)$.
\end{proof}

\begin{proposition}[Pointwise kernel sum bound]
\label{lem:auxcalc3}
Let $K^{(1)}, K^{(2)}, \dots$ be as in
Lemma \ref{lem:auxcalc2}. Then the sum
$$
\Phi_{\alpha, \beta}(t)
\coloneqq \sum_{m = 1}^\infty
K_{\alpha, \beta}^{(m)}(t)
$$
converges in $L^1([0,T])$ for each $(\alpha, \beta)
\in \Z^{d}\times\Z^d$.  With the notation of
\eqref{eq:K_ab_parts}, we also
have the pointwise estimate
\begin{equation*}
	\bigl|\Phi_{\alpha, \beta}(t)\bigr|
	\lesssim_T f(t,\alpha-\beta).
\end{equation*}
\end{proposition}

\begin{proof}
Let $C, C_1 > 0$ be the constants
given in Lemma \ref{lem:auxcalc2}.
Let $C_3$ be as given at the end of
the proof to Lemma \ref{lem:auxcalc2}.
From the pointwise bounds
\eqref{eq:Km_ab_bound}, for
$\alpha, \beta \in \Z^d$
and $t \in [0,T]$, we have
\begin{align*}
	\bigl|\Phi_{\alpha, \beta}(t)\bigr|
	&\le \sum_{m = 1}^\infty
	\bigl|K_{\alpha,\beta}^{(m)}(t)\bigr|
	\le \sum_{m = 1}^\infty
	\frac{C_3^m}{\Gamma(m/2)}t^{\bk{m-1}/2}
	f(t,\alpha-\beta).
\end{align*}
For $m \ge 0$,
$$
(m+1)\,\Gamma(m+1/2)\ge \Gamma(m+1).
$$
Hence, we can estimate
the sum  as follows:
\begin{align*}
	t^{-{1/2}}\sum_{m = 1}^\infty
	\frac{C_3^m}{\Gamma(m/2)} t^{m/2}
	& =t^{-{1/2}}\sum_{m = 1}^\infty
	\frac{C_3^{2m}}{\Gamma(m)} t^m
	+ t^{-{1/2}} \sum_{m = 0}^{\infty}
	\frac{C_3^{2m + 1}}{\Gamma(m + 1/2)} t^{m + 1/2}
	\\ &
	\le C_3^2 t^{{1/2}}\sum_{m = 0}^\infty
	\frac{C_3^{2m}}{\Gamma(m + 1)} t^m
	+ \sum_{m = 0}^{\infty}
	\frac{C_3^{2m + 1 } (m + 1)}{\Gamma(m + 1)} t^{m}
	\\ &
	= C_3^2 t^{{1/2}}\sum_{m = 0}^\infty
	\frac{C_3^{2m}}{\Gamma(m + 1)} t^m
	+C_3 t \sum_{m = 0}^{\infty}
	\frac{C_3^{2m} m}{\Gamma(m + 1)} t^{m - 1}
	\\* &
	\relspace + C_3  \sum_{m = 0}^{\infty}
	\frac{C_3^{2m}}{\Gamma(m + 1)} t^{m}
	\\ &
	=\bk{C_3^2 t^{{1/2}} +  C_3^3 t + C_3} e^{C_3^2 t}.
\end{align*}
Therefore, in the notation of \eqref{eq:K_ab_parts},
\begin{align*}
	\bigl|\Phi_{\alpha, \beta}(t)\bigr|
	\le \bk{C_3^2 T^{\hf}+C_3^3 T + C_3} e^{C_3^2 T}
	f(t,\alpha-\beta).
\end{align*}
Since the series is absolutely convergent
pointwise in $\alpha, \beta, t$, and is
majorised by an element in $L^1([0,T])$
for each $(\alpha, \beta)$, the dominated
convergence theorem implies the proposition.
\end{proof}



\begin{thebibliography}{10}

\bibitem{AD2022}
L.~Abadias and M.~De Le\'on-Contreras.
\newblock Discrete {H}\"older spaces and their characterization via semigroups
  associated with the discrete {L}aplacian and kernel estimates.
\newblock {\em J. Evol. Equ.}, 22(4):Paper No. 91, 42, 2022.

\bibitem{AGMP2021}
L.~Abadias, J.~Gonz\'alez-Camus, P.~J. Miana, and J.~C. Pozo.
\newblock Large time behaviour for the heat equation on {$\mathbf{Z}$}, moments
  and decay rates.
\newblock {\em J. Math. Anal. Appl.}, 500(2):Paper No. 125137, 25, 2021.

\bibitem{AGR2023}
L.~Abadias, J.~Gonz\'alez-Camus, and S.~Rueda.
\newblock{Time-step heat problem on the mesh: asymptotic behavior and decay rates.}
\newblock{\em Forum Math.}, 35(6): 1563–1582, 2023.

\bibitem{ADM2024}
L.~Abadias, M.~D. Le{\'o}n-Contreras, and A.~Mahillo.
\newblock End-point maximal regularity for the discrete parabolic
Cauchy problem and regularity of non-local operators in
discrete Besov spaces.
\newblock{\em J. Differential Equations}, 440:Paper 113465, 2025.

\bibitem{BW2016}
R.~Beals and R.~Wong.
\newblock {\em Special functions and orthogonal polynomials}, volume 153 of
  {\em Cambridge Studies in Advanced Mathematics}.
\newblock Cambridge University Press, Cambridge, 2016.

\bibitem{CJKS2023}
G.~Chinta, J.~Jorgenson, A.~Karlsson, and L.~Smajlovi{\'c}.
\newblock The parametrix construction of the heat kernel on a graph, arXiv
  2308.04174.

\bibitem{CGRTV2017}
O.~Ciaurri, T.~A. Gillespie, L.~Roncal, J.L.~Torrea, and J.L.~Varona.
\newblock Harmonic analysis associated with a discrete {L}aplacian.
\newblock {\em J. Anal. Math.}, 132:109--131, 2017.

\bibitem{Davies:1989ab}
E.B.~Davies.
\newblock {\em Heat Kernels and Spectral Theory}.
\newblock Cambridge Tracts in Mathematics. Cambridge University Press, 1989.

\bibitem{Feller:1966aa}
W.~Feller.
\newblock {\em An Introduction to Probability Theory and Its Applications},
  volume~2.
\newblock Wiley, 1966.

\bibitem{Fjordholm:2023aa}
U.~S. Fjordholm, K.~H. Karlsen, and P.~H. Pang.
\newblock Convergent finite difference schemes for stochastic transport
  equations.
\newblock {\em SIAM J. Numer. Anal.}, 63(1):149--192, 2025.

\bibitem{Fri1964}
A.~Friedman.
\newblock {\em Partial differential equations of parabolic type.}
\newblock Prentice-Hall, Inc., Englewood Cliffs, N.J.,, 1964.

\bibitem{Grigoryan:2009ab}
A.~Grigor'yan.
\newblock {\em Heat kernel and analysis on manifolds}, volume~47 of {\em AMS/IP
  Studies in Advanced Mathematics}.
\newblock American Mathematical Society, Providence, RI; International Press,
  Boston, MA, 2009.

\bibitem{Hoff:1985zm}
D.~Hoff and J.~Smoller.
\newblock Error bounds for finite-difference approximations for a class of
  nonlinear parabolic systems.
\newblock {\em Math. Comp.}, 45(171):35--49, 1985.

\bibitem{Ignat:2006aa}
L.I.~Ignat.
\newblock Qualitative properties of a numerical scheme for the heat equation.
\newblock In {\em Numerical mathematics and advanced applications}, pages
  593--600. Springer, Berlin, 2006.

\bibitem{Jorgenson:2024aa}
J.~Jorgenson, A.~Karlsson, and L.~Smajlovi{\'c}.
\newblock Constructing heat kernels on infinite graphs, arXiv 2404.11535.

\bibitem{Lan2005}
C.~Landim.
\newblock Gaussian estimates for symmetric simple exclusion processes.
\newblock {\em Ann. Fac. Sci. Toulouse Math. (6)}, 14(4):683--703, 2005.

\bibitem{Lizama:2024ab}
C.~Lizama and M.~Warma.
\newblock The semi-discrete diffusion convection equation with decay.
\newblock {\em Appl. Math. Lett.}, 151:Paper No. 108979, 5, 2024.

\bibitem{Pang:1993aa}
M.M.H.~Pang.
\newblock Heat kernels of graphs.
\newblock {\em J. London Math. Soc. (2)}, 47(1):50--64, 1993.

\bibitem{Slavik:2020aa}
A.~Slav\'ik.
\newblock Asymptotic behavior of solutions to the semidiscrete diffusion
  equation.
\newblock {\em Appl. Math. Lett.}, 106:106392, 7, 2020.

\bibitem{Slavik:2022aa}
A.~Slav\'ik.
\newblock Asymptotic behavior of solutions to the multidimensional semidiscrete
  diffusion equation.
\newblock {\em Electron. J. Qual. Theory Differ. Equ.}, pages Paper No. 9, 9,
  2022.

\bibitem{Taylor:2011aa}
M.E.~Taylor.
\newblock {\em Partial differential equations {II}. {Q}ualitative studies of
  linear equations}, volume 116 of {\em Applied Mathematical Sciences}.
\newblock Springer, New York, second edition, 2011.

\bibitem{Varopoulos:1992aa}
N.~T. Varopoulos, L.~Saloff-Coste, and T.~Coulhon.
\newblock {\em Analysis and geometry on groups}, volume 100 of {\em Cambridge
  Tracts in Mathematics}.
\newblock Cambridge University Press, Cambridge, 1992.

\end{thebibliography}

\end{document}